\documentclass[11pt,reqno,twoside]{article}
\usepackage{mysty_arxiv}

\usepackage{wrapfig}
\usepackage{mdwlist}

\usepackage{fullpage}
\usepackage{patchcmd,pifont}


\usepackage{enumitem}

\SetLabelAlign{parright}{\parbox[t]{\labelwidth}{\raggedleft{#1}}}
\setlist[description]{style=multiline,topsep=4pt,align=parright}

\makeatletter
\let\reftagform@=\tagform@
\def\tagform@#1{\maketag@@@{(\ignorespaces\textcolor{black}{#1}\unskip\@@italiccorr)}}
\newcommand{\iref}[1]{\textup{\reftagform@{\tcr{\ref{#1}}}}}
\makeatother

\begin{document}
\title{Stochastic Inertial Dynamics Via Time Scaling and Averaging}
\author{Rodrigo Maulen-Soto\thanks{Normandie Universit\'e, ENSICAEN, UNICAEN, CNRS, GREYC, France. E-mail: rodrigo.maulen@ensicaen.fr} \and
Jalal Fadili\thanks{Normandie Universit\'e, ENSICAEN, UNICAEN, CNRS, GREYC, France. E-mail: Jalal.Fadili@ensicaen.fr} \and Hedy Attouch\thanks{IMAG, CNRS, Universit\'e Montpellier, France. E-mail: hedy.attouch@umontpellier.fr} \and Peter Ochs\thanks{Department of Mathematics and Computer Science, Saarland University, Germany, E-mail: ochs@math.uni-sb.de}
}
\date{}
\maketitle

\begin{abstract}
Our work is part of the close link between continuous-time dissipative dynamical systems and optimization algorithms, and more precisely here, in the stochastic setting. We aim to study stochastic convex minimization problems through the lens of stochastic inertial differential inclusions that are driven by the subgradient of a convex objective function. This will provide a general mathematical framework for analyzing the convergence properties of stochastic second-order inertial continuous-time dynamics involving vanishing viscous damping and measurable stochastic subgradient selections. Our chief goal in this paper is to develop a systematic and unified way that transfers the properties recently studied for first-order stochastic differential equations to second-order ones involving even subgradients in lieu of gradients. This program will rely on two tenets: time scaling and averaging, following an approach recently developed in the literature by one of the co-authors in the deterministic case. Under a mild integrability assumption involving the diffusion term and the viscous damping, our first main result shows that almost surely, there is weak convergence of the trajectory towards a minimizer of the objective function and fast convergence of the values and gradients. We also provide a comprehensive complexity analysis by establishing several new pointwise and ergodic convergence rates in expectation for the convex, strongly convex, and (local) Polyak-{\L}ojasiewicz case. Finally, using Tikhonov regularization with a properly tuned vanishing parameter, we can obtain almost sure strong convergence of the trajectory towards the minimum norm solution.
\end{abstract}



\section{Introduction}\label{sec:intro}

\subsection{Problem Statement}\label{problem_statement}

We consider the minimization problem
\begin{equation}\label{P}\tag{P}
    \min_{x\in \H} F(x)\eqdef f(x)+g(x),
\end{equation}
where $\H$ is a separable real Hilbert space, and the objective $F$ satisfies the following standing assumptions:
\begin{align}\label{H0}
\begin{cases}
\text{$f:\H\rightarrow\R$ is continuously differentiable and convex with $L$-Lipschitz continuous gradient}; \\
\text{$g:\H\rightarrow\R\cup\{\pm\infty\}$ is proper, lower semi-continuous (lsc) and convex};\\
\calS_F\eqdef \argmin (F)\neq\emptyset. \tag{$\mathrm{H}_0$} 
\end{cases}
\end{align} 

\tcm{
\paragraph{First-order in time systems.}
To solve \eqref{P} when $g\equiv 0$, a fundamental dynamic is the gradient flow system:
\begin{equation}
\begin{cases}\label{gf}\tag{GF}
\begin{aligned}
\dot{x}(t)+\nabla f(x(t))&= 0,\quad t>t_0;\\
x(t_0)&=x_0.
\end{aligned}
\end{cases}
\end{equation}
The gradient system \eqref{gf} is a dissipative dynamical system, whose study dates back to Cauchy \cite{Cauchy} in finite dimension. It plays a fundamental role in optimization: it transforms the problem of minimizing $f$ into the study of the asymptotic behavior of the trajectories of \eqref{gf}. This example was the precursor to the rich connection between continuous dissipative dynamical systems and optimization. It is well known since the founding papers of Brezis, Baillon, Bruck in the 1970s that, if the solution set $\argmin(f)$ of \eqref{P} is non-empty, then each solution trajectory of \eqref{gf} converges weakly, and its (weak) limit belongs to $\argmin (f)$. Moreover, this dynamic is known to yield a convergence rate of $\mathcal{O}(t^{-1})$ (in fact even $o(t^{-1})$) on the values.

The Euler forward (a.k.a. Euler-Maruyama) discretization \eqref{gf}, with stepsize sequence $h_k>0$, is the celebrated gradient descent scheme 
\begin{equation}\label{GD}\tag{GD}
    x_{k+1}=x_k-h_k \nabla f(x_k).
\end{equation} 
Under \eqref{H0}, and for $(h_k)_{k \in \N} \subset ]0,2/L[$, then we have both the convergence of the values \linebreak $f(x_k)- \min f=\calO(1/k)$ (in fact even $o(1/k)$), and the weak convergence of iterates $(x_k)_{k \in\N} $ to a point in $\argmin (f)$.  This convergence rate can be refined under various additional geometrical properties on the objective $f$ such as error bounds (and the closely related Kurdyka-{\L}ojasiewicz property in the convex case; see \cite{BolteKLComplexity16}).

\paragraph{Second-order in time systems: Key role of inertia.}
Second-order in time inertial dynamical systems have been introduced to provably accelerate the convergence behavior compared to \eqref{gf}. An abundant literature has been devoted to the study of inertial dynamics 
\begin{equation}\tag{$\mathrm{IGS}_{\gamma}$}\label{AVD}
\ddot{x}(t)+  \gamma (t)\dot{x}(t)  +\nabla f(x(t))=0,\quad t>t_0 .
\end{equation}
The importance of working with an asymptotically vanishing viscosity coefficient $\gamma(t)$ to obtain acceleration was stressed by several authors; see for instance \cite{cabot, engler}. Most of the literature focuses on the case $\gamma (t) =\frac{\alpha}{t}$, originating from the seminal work of \cite{su} who showed the rate of convergence $\mathcal O(1/t^2)$ of the values for $\alpha=3$, thus making the link with the Nesterov accelerated gradient method \cite{1983}. Since then, an important body of literature has been devoted to this important case and the subtle tuning of the parameter $\alpha$. Taking $\alpha \geq 3$ is mandatory for provable accelerated rate $\mathcal O \left(1/t^2 \right)$ of the values \cite{11}, and $\alpha > 3$ provides an even better rate $o \left(1/t^2 \right)$ and weak convergence of the trajectory \cite{faster1k2,may}. On the other hand, $\alpha < 3$ necessarily leads to a slower rate $\mathcal O \left(1/t^{2\alpha/3} \right)$ \cite{Apidopoulos18,ACR-subcrit}. Another remarkable instance of \eqref{AVD} corresponds to the well-known Heavy Ball with Friction (HBF) method, where $\gamma(t)$ is a constant, first introduced by Polyak in \cite{speedingup}. When the objective is strongly convex, it was shown that the trajectory converges exponentially with an optimal convergence rate for a properly chosen constant $\gamma$ \cite{heavyb}. 

\paragraph{Geometric Hessian-driven damping.}
Because of the inertial aspects, and the asymptotic vanishing viscous damping coefficient, \eqref{AVD} may exhibit many small oscillations which are not desirable from an optimization point of view. To remedy this, a powerful tool consists in introducing a geometric damping driven by the Hessian of $f$ into the dynamic. This yields the Inertial System with Explicit Hessian-driven Damping \cite{APR2,ACFR}
\begin{equation} \label{ISEHD}\tag{ISEHD}
\qquad \ddot{x}(t) + \gamma(t)\dot{x}(t) +  \beta (t) \nabla^2  f (x(t)) \dot{x} (t) + \nabla  f (x(t)) = 0,
\end{equation}
and the Inertial System with Implicit Hessian-driven Damping \cite{alecsa,MJ}
\begin{equation}
\begin{cases}\label{ISIHD}\tag{ISIHD}
\ddot{x}(t)+\gamma(t)\dot{x}(t)+\nabla f(x(t)+\beta(t)\dot{x}(t)) = 0 .
\end{cases}
\end{equation}
The rationale behind the use of the term ``implicit'' in \eqref{ISIHD} comes from a Taylor expansion of the gradient term (as $t \to +\infty$ we expect $\dot{x}(t) \to 0$). Following the physical interpretation of these ODEs, we call the non-negative parameters $\gamma$ and $\beta$ as the viscous and geometric damping parameters, respectively. 

At first glance, the presence of the Hessian in \eqref{ISEHD} may seem to entail numerical difficulties. However, this is not the case as the term $\nabla^2  f (x(t)) \dot{x} (t)$ is nothing but the time derivative of $t \mapsto \nabla  f (x(t))$. This explains why the time discretization of this dynamic provides efficient first-order algorithms \cite{ACFR}. On the other hand, \eqref{ISEHD} can be argued to be truly of second-order nature, i.e., close to Newton's and Levenberg-Marquardt's dynamics \cite{castera}. This understanding suggests that \eqref{ISIHD} may reflect the nature of first-order algorithms more faithfully than \eqref{ISEHD}. {Moreover, when it will come to our stochastic setting, which is the focus of this paper, the approach we propose will only make sense for the implicit form of the Hessian-driven damping. Indeed, in our stochastic setting, we do not have direct access to the evaluation of the gradient of $f$. Instead, we model the associated errors with a continuous It\^o martingale (denoted as $M(t)$). But for the explicit form of the Hessian-driven damping, this would entail a term involving the time derivative of $\nabla f ({X}(t))+M(t)$. This is meaningless as (non-constant) martingales are not differentiable a.s.. This is why from now on, we will solely consider \eqref{ISIHD}.
}

\subsection{Motivations}
In the following, $\H,\K$ are real separable Hilbert spaces. In many practical situations, the (sub-)gradient evaluation is subject to stochastic errors. This is for example the case if the cost per iteration is very high and thus cheap and random approximations of the (sub-)gradient are necessary. These errors can also be due to some other exogenous factor. The  continuous-time approach through stochastic differential equations (SDE) is a powerful way to model these errors in a unified way, and stochastic algorithms can then be viewed as time- discretizations. In fact, several recent works have used the SDE perspective to model SGD-type algorithms motivated by various reasons; (see \eg \cite{optch,continuous,dif,schrodinger,valid,cycle,mertikopoulos_staudigl_2018,mio,sdemodel,mio2}). The continuous-time perspective offers a deep insight and unveils the key properties of the dynamic without being tied to a specific discretization.

In this context, the first SDE that comes to mind is
\begin{equation}
    \label{CSGDintro}\tag{SGF}
    \begin{cases}
     dX(t)&=-\nabla f(X(t))+\sigma(t,X(t))dW(t), \quad t>t_0;\\
    X(t_0)&=X_0,
    \end{cases}
\end{equation}
defined over a complete filtered probability space $(\Omega,\calF, (\calF_t)_{t \geq t_0}, \mathbb{P})$, where the diffusion (volatility) term $\sigma: \R^+ \times \H \to \mathcal{L}_2(\K;\H)$ is a measurable function, $W$ is a $\calF_t$-adapted $\K$-valued Brownian motion, and the initial data $X_0$ is an $\calF_{t_0}-$measurable $\H$-valued random variable. All these notations and concepts are introduced in Section~\ref{sec:notation}. \eqref{CSGDintro} is a stochastic counterpart of \eqref{gf} where the error is modeled as a stochastic integral with respect to the measure defined by a continuous It\^{o} martingale. 

\paragraph{SDE modeling of SGD.}\label{sec:sdesgd}
To simplify the discussion, let us focus in this section on the finite-dimensional case ($\H=\R^d$, $\K=\R^m$). In various areas of science and engineering, in particular in machine learning, the Stochastic Gradient Descent (SGD) is a powerful alternative to gradient descent, and consists in replacing the full gradient computation by a cheap random version, serving as an unbiased estimator. The SGD updates the iterates according to
\begin{equation}\tag{SGD}\label{sgdd}
x_{k+1}=x_k-h(\nabla f(x_k) + e_k)
\end{equation}
where $h \in \R_+$ is the stepsize and $e_k$ is the random noise term on the gradient at the $k$-th iteration. As such, \eqref{sgdd} can be viewed as instance of the Robbins-Monro stochastic approximation algorithm \cite{rob}. When the objective takes the form $f(x) = \EE[\hat{f}(x,\xi)]$, where the expectation is w.r.t. to the random variable $\xi$ the single-batch version of SGD reads
\begin{equation}\label{sgdmini}\tag{$\mathrm{SGD_{SB}}$}
x_{k+1}=x_k-h\nabla \hat{f}(x_k,\xi_k),
\end{equation}
where $\seq{\xi_k}$ are i.i.d. random variables with the same distribution as $\xi$. Of course, \eqref{sgdmini} is an instance of \eqref{sgdd} with $e_k = \nabla \hat{f}(x_k,\xi_k) - \nabla f(x_k)$.




{The SDE continuous-time approach is motivated by its close relation to \eqref{sgdd} or \eqref{sgdmini}. We first note that when the noise $e_k$ in \eqref{sgdd} is $\mathcal{N}(0,\sigma_k I_d)$, \eqref{CSGDintro} is a better continuous-time model for \eqref{sgdd} than \eqref{gf}, as has been shown recently in \cite[Proposition~2.1]{sdemodel}. There, the sequence $\seq{x_k}$ provided by \eqref{sgdd}, with $e_k\sim\mathcal{N}(0,\sigma_k I_d)$, was proved to be accurately approximated by \eqref{CSGDintro} with $\sigma(t,X(t)) = \sqrt{h}\sigma(t)$ and $\sigma(kh)=\sigma_kI_d$. 
}

{For the standard single-batch SGD \eqref{sgdmini}, the argument is more involved. Actually, many recent works (see \eg \cite{Mandt16,optch,continuous,dif,schrodinger,valid,cycle,Xie21,latz,hessianaware}) have linked algorithm \eqref{sgdmini} with continuous-time first-order stochastic diffusion dynamics such as \eqref{CSGDintro}. These works show either empirically or theoretically under which conditions (appropriate drift and diffusion terms, regularity of $f$, etc.) \eqref{CSGDintro} can be seen as a good approximation model of \eqref{sgdmini} for fixed stepsize. By a good model we mean that \eqref{CSGDintro} is a continuum limit of \eqref{sgdmini} as the stepsize $h$ goes to zero, or equivalently that the approximation of \eqref{sgdmini} via the diffusion process \eqref{CSGDintro} is precise in some weak sense; see \cite{optch,dif,valid}. }

{As a consequence, using \eqref{CSGDintro} as a proxy of \eqref{sgdd} or \eqref{sgdmini} allows to capitalize on the wealth of results in the field of SDEs, It\^o calculus and measure theory, and this in turn opens the door to new insights in the behavior of \eqref{sgdd} or \eqref{sgdmini} and to transfer all the convergence results that one can prove for \eqref{CSGDintro} to \eqref{sgdd}. Actually, this is one of the main messages we want to convey in this work. Our motivation and results are also complementary to those in the literature. Indeed, most, if not all, of works cited in the previous paragraph are primarily motivated by the fact that continuous-time SDE approximation of \eqref{sgdd} is a crucial tool to study its escape behavior of bad saddle points (a.k.a. traps) in the non-convex smooth case. Our standpoint, which is line with \cite{mio,mertikopoulos_staudigl_2018,sdemodel,mio2}, is complementary and we argue that the continuous-time perspective offers a deep insight and unveils the key properties of the dynamic, without being tied to a specific discretization. This enlightens the behavior of the sequence generated by some specific algorithm. In turn, studying the continuous-time SDE will allow to predict the convergence behavior of stochastic algorithms seen as a discretization of the corresponding continuous-time dynamics.} 

\paragraph{Second-order SDE modeling of inertial SGD.}
Using a lifting argument to get an equivalent first-order reformulation, a natural generalization of \eqref{ISIHD} to the non-smooth case yields the differential inclusion
\begin{equation}
\begin{cases}\label{ISIHDN}\tag{$\mathrm{ISIHD_{NS}}$}
\begin{aligned}
\dot{x}(t)&=v(t), \quad t>t_0; \\
\dot{v}(t)&\in -[\gamma(t)v(t)+\partial F(x(t)+\beta(t)v(t))], \quad t>t_0;\\
x(t_0)&=x_0,\quad \dot{x}(t_0)=v_0,
\end{aligned}
\end{cases}
\end{equation}
where $\partial F$ is the convex subdifferential of $F$. In this setting, keeping in mind that we want to give a rigorous meaning to \eqref{ISIHDN}, we can model the associated errors using a stochastic integral with respect to the measure defined by a continuous It\^o martingale. This entails the following stochastic differential inclusion (SDI for short), which is the stochastic counterpart of \eqref{ISIHDN}:
\begin{equation}\label{ISIHDN-S}\tag{$\mathrm{S-ISIHD_{NS}}$}
\begin{cases}
d{X(t)}&={V(t)}dt,  \nonumber\\
d{V(t)}&\in -\gamma(t){V(t)}dt-\partial F({X(t)}+\beta(t){V(t)})dt+\sigma(t,{X(t)}+\beta(t){V(t)})dW(t),\\
{X(t_0)}&={X_0}, \quad {V(t_0)}={V_0}. \nonumber
\end{cases}
\end{equation}
When $g \equiv 0$, we recover the stochastic counterpart of \eqref{ISIHD} as the following SDE
\begin{equation}\label{ISIHD-S}\tag{$\mathrm{S-ISIHD}$}
\begin{cases}
d{X(t)}&={V(t)}dt,  \nonumber\\
d{V(t)}&= -\gamma(t){V(t)}dt-\nabla f({X(t)}+\beta(t){V(t)})dt+\sigma(t,{X(t)}+\beta(t){V(t)})dW(t),\\
{X(t_0)}&={X_0}, \quad {V(t_0)}={V_0} .
\end{cases}
\end{equation}
In the smooth but non-convex case, and motivated again by strict saddle points avoidance, the authors in \cite{Hu17,hbnoise} study the convergence behavior of a stochastic discrete heavy ball method from its approximating SDE. The latter is a randomly perturbed nonlinear oscillator in \cite{Hu17} and a coupled system of nonlinear oscillators in \cite{hbnoise}. These SDEs are very different from the ones considered here. 

Extending what we have discussed above for \eqref{CSGDintro} as a good model of \eqref{sgdd}, we will show in Proposition~\ref{prop:isgfconsist} (see Section~\ref{sec:sdeode} for details) that \eqref{ISIHD-S} is a good model of the natural stochastic inertial algorithm \eqref{stochnest} obtained by simple Euler forward (a.k.a. Euler-Maruyama) discretization of \eqref{ISIHD-S}. The corresponding convergence rate as a function of the time step-size $h$ is $\mathcal{O}(h)$ and show that this is much better that the approximation rate of \eqref{ISIHD} which is only $\mathcal{O}(\sqrt{h})$. As a consequence, this justifies and motivates the continuous-time dynamics \eqref{ISIHD-S} (and \eqref{ISIHDN-S}) as a good proxy of stochastic inertial algorithms and opens the door to new insights in the behavior of such algorithms, and that their convergence properties can be easily derived from those of \eqref{ISIHD-S} with minimal effort. For instance, when $f$ is smooth with Lipschitz-continuous gradient, it is easy to see from the descent lemma that
\[
\EE[f(X_k) - \min f] = \EE[f(X(kh)) - \min f] + \calO(h),
\]
where $X_k$ are the iterates of \eqref{stochnest} and $X(t)$ the solution to \eqref{ISIHD-S}. This means that any rate proved on $\EE[f(X(t))-\min f]$ can be directly transferred to $\EE[f(X_k)-\min f]$.
}

\tcm{
\subsection{Objectives and Contributions}
In this work, our goal is to provide a general mathematical framework for analyzing the convergence properties of \eqref{ISIHD-S} and \eqref{ISIHDN-S}. In this context, considering inertial dynamics with a time-dependent vanishing viscosity coefficient $\gamma$ is a key ingredient to obtain fast convergent methods. We will develop a systematic and unified way that transfers the properties of stochastic first-order dynamics recently studied by  \cite{mio,mio2} to second-order ones. Our program will then rely on two pillars: {\textit{time scaling}} and {\textit{averaging}}, following the methodology recently developed in \cite{fast} for the deterministic case.

More precisely, we study the stochastic dynamics \eqref{ISIHDN-S} and its long-time behavior in order to solve \eqref{P}. We conduct a new analysis using specific and careful arguments that are much more involved than in the deterministic case. To get some intuition, keeping the discussion informal at this stage, let us first identify the assumptions needed to expect that the position state of \eqref{ISIHD-S} ``approaches'' $\argmin (f)$ in the long run. In the case where $\H=\K$, $\gamma(\cdot)\equiv\gamma>0,\beta\equiv 0,$ and $\sigma=\tilde{\sigma} I_{\H}$, where $\tilde{\sigma}$ is a positive real constant. Under mild assumptions one can show that \eqref{ISIHD-S} has a unique invariant distribution $\pi_{\tilde{\sigma}}$ in $(x,v)$ with density proportional to $ \exp\pa{-\frac{2\gamma}{\tilde{\sigma}^2}\left(f(x)+\frac{\Vert v\Vert^2}{2}\right)}$, see e.g., \cite[Proposition~6.1]{pavliotis}. 
Clearly, as $\tilde{\sigma} \to 0^+$, $\pi_{\tilde{\sigma}}$ gets concentrated around $\argmin (f)\times \{0_{\H}\}$, with $\lim_{\tilde{\sigma} \to 0^+} \pi_{\tilde{\sigma}}(\argmin (f) \times  \{0_{\H}\}) = 1$; see also Section~\ref{subsec:relwork} for further discussion. Motivated by these observations and the fact that we aim to exactly solve \eqref{P}, our paper will then mainly focus on the case where $\sigma(\cdot,x)$ vanishes fast enough as $t \to +\infty$ uniformly in $x$.
}

Our main contributions are summarized as follows:
\begin{itemize}
\item We show almost sure weak convergence of the trajectory (see Theorem~\ref{trajectory}) and convergence rates (see Theorem~\ref{maxai}) in expectation in the case of time-dependent coefficients $\gamma(t)$ and a proper choice of $\beta(t)$ (depending on $\gamma$). For this analysis, we transfer the results from the Lyapunov analysis of the first-order in-time stochastic (sub-)gradient system studied in \cite{mio,mio2} from which our inertial system is built through time scaling and averaging.
\item We obtain almost sure and ergodic convergence results which correspond precisely to the best-known results in the deterministic case. In particular, if we let $\alpha>3, \gamma(t)=\frac{\alpha}{t}, \beta(t)=\frac{t}{\alpha-1}$, then under appropriate assumptions on the diffusion (volatility) term $\sigma$, we obtain the rate of convergence $o(1/t^2)$ of the values as well as fast convergence of the gradient in almost sure sense (see Corollary~\ref{corat}). \tcm{This corresponds to the best known results of second-order in-time Hessian driven damping dynamics in the deterministic case. 
\item As far as the interplay between viscous damping, geometric damping, and noise level, our results reveal that viscous damping, which allows for acceleration, enters in the control of the noise and may impose a more stringent summability condition on the diffusion coefficient (see e.g., Theorem~\ref{trajectory}, Theorem~\ref{maxai} and Corollary~\ref{corat}). The geometric damping, which is designed to reduce oscillations, does not enter directly in this control and can be chosen in a more flexible way.}
\item We then turn to providing a local analysis with a local linear convergence rate under the Polyak-{\L}ojasiewicz inequality (See Theorem \ref{crpolyak}). This is much more challenging in the stochastic case, and even more for second-order systems, as localizing the process in this case is very delicate. \tcm{Leveraging time scaling and averaging offers an elegant framework to achieve such a local convergence analysis while it is barely possible otherwise. To the best of our knowledge, this is the first result of this kind for these systems in the literature.}
\item We also show almost sure strong convergence of the trajectory to the minimal norm solution when adding a Tikhonov regularization to our systems (see Theorem \ref{trajectory1}). Moreover, we show convergence rates in expectation for the objective and the trajectory for a particular Tikhonov regularizer (see Theorem~\ref{importante1}).
\end{itemize}

It is worth observing that since our approach is based on an averaging technique, it will involve Jensen's inequality at some point to get fast convergence rates. In this respect, the convexity assumption on the objective function appears unavoidable, at least in this proof strategy. It is also worth stressing the fact that the approach only makes sense for the implicit form of the Hessian-driven damping. {Indeed, as explained above, the explicit form of the Hessian-driven damping has a term involving the time derivative of the (sub)gradient at the trajectory. As the noise, modeled here as an It\^o martingale, in practice stems from the (sub)gradient evaluation, this time derivative is meaningless with explicit Hessian-driven damping, as (non-constant) martingales are a.s. non-differentiable.}


\subsection{Relation to prior work}\label{subsec:relwork}

\paragraph{Kinetic diffusion dynamics for sampling.} Let us consider \eqref{ISIHD-S} in the case where $\H=\K=\R^n$, $\gamma(t)=\gamma>0, \beta\equiv 0$, and $\sigma=\sqrt{2\gamma}I$. Then one recovers the kinetic Langevin diffusion (or second-order Langevin process). In this case, the continuous-time Markov process $(X(t),V(t))$ is positive recurrent and has a unique invariant distribution which has the density $\propto \exp\left(-f(x)-\frac{\|v\|^2}{2}\right)$ with respect to the Lebesgue measure on $\R^{2n}$. Time-discretized versions of this Langevin diffusion process have been studied in the literature to (approximately) sample from $\propto \exp(-f(x))$ with asymptotic and non-asymptotic convergence guarantees in various topologies and under various conditions; see \cite{underdamped,isthere,Dalalyan2019BoundingTE} and references therein. \tcm{By rescaling the problem, relation between sampling and optimization with quantitive estimates has been investigated in \eg \cite{user} for the strongly convex case.}


\tcm{
\paragraph{First-order stochastic (sub)gradient systems.} In \cite{mio} (resp. \cite{mio2}), the authors studied \eqref{CSGDintro} (resp. its non-smooth version as an SDI) as a proxy for \eqref{sgdd}. One of their main results is the almost sure weak convergence of the trajectory to the set of minimizers under integrability conditions on $\sigma$, as well as convergence rates in expectation and in almost sure sense of the dynamic under different geometries of $f$. Our goal here is to take these results to second-order inertial dynamics featuring both viscous and geometric dampings. This turns out to be a challenging task. We will show that the convergence rate on the objective can be achieved provided that the noise vanishes sufficiently fast.
}

\paragraph{Inexact inertial gradient systems.} There is an abundant literature regarding the dynamics \eqref{ISIHD} and \eqref{ISEHD}, either in the exact case or with errors but only deterministic ones; see \cite{alecsa, hessianpert,27,35,AttouchCabot18,11,20,34,37,ABCR,ACFR,10, casterainertial}. 
Only a few papers have been devoted to studying stochastic versions of \eqref{AVD}, either with vanishing damping $\gamma(t)=\alpha/t$ or constant damping $\gamma(t)$ (stochastic HBF); see e.g. \cite{gadat1, Gadat2018, sdemodel, rolememory}. \tcm{One of the advantages of considering the stochastic version of \eqref{ISIHD} is the possible reduction of oscillations thanks to the introduction of the geometric damping. We show that this effect could be preserved in the stochastic setting, hence obtaining more stable trajectory solutions than other stochastic second-order dynamics. Additionally, having the flexibility to choose the viscous damping would eventually allow us to achieve faster convergence results}. For a stochastic version of \eqref{AVD}, \cite{sdemodel} provide asymptotic $\mathcal{O}(1/t^2)$ convergence rate on the objective values in expectation under integrability conditions on the diffusion term as well as other rates under additional geometrical properties of the objective. These geometrical assumptions come in the form of {\textit{global}} growth and flatness of the objective which is very restrictive. Rather, here, our geometrical assumptions on $f$ will be only local. The corresponding stochastic algorithms for these two choices of $\gamma$, whose mathematical formulation and analysis are simpler, have been the subject of active research work; see \eg \cite{Lin2015,Frostig2015,Jain2018,AR,AZ,Yang2016,Gadat2018,Loizou2020,Laborde2020,LanBook,Sebbouh2021,Driggs22,Wu24,AFKstochastic24}.

\paragraph{Time scaling and averaging.} 
The authors in \cite{fast} proposed time scaling and averaging to link \eqref{gf} and \eqref{ISIHD} with a general viscous damping function $\gamma$ and a properly adjusted geometric damping function $\beta$ (related to $\gamma$). Our aim is to extend the results of \cite{fast} to the stochastic case. Leveraging these techniques with a general function $\gamma$ and an appropriate $\beta$, we will be able to transfer all the results we obtained in \cite{mio2} for a first-order SDI to the second-order one \eqref{ISIHDN-S}. This avoids in particular to go through an intricate and a dedicated Lyapunov analysis for \eqref{ISIHDN-S}. A local convergence analysis becomes also easily accessible through this lens while it is barely possible otherwise. We also specialize our results to the standard case where $\gamma(t)=\frac{\alpha}{t}$ and $\beta(t)=\frac{t}{\alpha-1}$. 



\smallskip

The idea of passing from a first-order system to a second-order one via time scaling is not new. In the smooth case ($g \equiv 0$), the author of \cite{cabott} propose time scaling and a tricky change of variables to show that \eqref{AVD} is equivalent to an averaged gradient system, \ie the steepest gradient system \eqref{gf} where the instantaneous value of $\nabla f(x(t))$ is replaced by some average of the gradients $\nabla f(x(s))$ over all past positions $s \leq t$. See also \cite{Goudou05} for more general gradient systems with memory terms involving kernels. This gives rise to an integro-differential equation whose asymptotic behavior and the equivalent second-order dynamic have been investigated in \cite{cabott}. A stochastic version of this integro-differential equation has been studied in \cite{gadat1} where the long time behavior of the resulting process, in particular its invariant distribution and occupation measure, was investigated under ellipticity assumptions on $f$ and $\sigma$ and proper behavior of the averaging gradient function. Clearly, the motivation of that work is not on the minimizing properties of the process while it is our focus here.


\subsection{Organization of the paper}
Section~\ref{sec:notation} introduces notations, definitions and preliminaries that are essential to our exposition. Section~\ref{sec:scaling} is the main part of our study. We develop the passage from the first-order system to the second-order inertial system by using the time scaling and averaging in a stochastic framework. Almost sure and ergodic convergence rates are provided under different geometric properties of the objective function, such as convexity and Polyak-{\L}ojasiewicz geometry. Finally, we show a strong convergence result when adding a Tikhonov regularization. Additional technical results that are needed throughout the paper are gathered in Appendix~\ref{aux}.

\section{Notation and Preliminaries}\label{sec:notation}
\subsection{Notation}
We will use the following shorthand notations:  Given $n\in\N$,  $[n]\eqdef \{1,\ldots,n\}$. {Consider $\H,\K$ real separable Hilbert spaces endowed with the inner product $\langle\cdot,\cdot\rangle_{\H}$ and $\langle\cdot,\cdot\rangle_{\K}$, respectively, and norm $\Vert \cdot\Vert_{\H}=\sqrt{\langle \cdot,\cdot \rangle_{\H}}$ and $\Vert \cdot\Vert_{\K}=\sqrt{\langle \cdot,\cdot \rangle_{\K}}$, respectively (we will omit the subscripts $\H$ and $\K$ whenever it is clear from the context). $I_{\H}$ is the identity operator from $\H$ to $\H$. $\calL(\K;\H)$ is the space of bounded linear operators from $\K$ to $\H$, $\calL_1(\K)$ is the space of trace-class operators, and $\calL_2(\K;\H)$ is the space of bounded linear Hilbert-Schmidt operators from $\K$ to $\H$}. For $M\in\calL_1(\K)$, the trace is defined by
\[
\tr(M)\eqdef \sum_{i\in I} \langle Me_i,e_i\rangle<+\infty,
\]
where $I\subseteq \N$ and $(e_i)_{i\in I}$ is an orthonormal basis of $\K$. Besides, for $M\in\calL(\K;\H)$, $M^{\star}\in\calL(\H;\K)$ is the adjoint operator of $M$, and for $M\in\calL_2(\K;\H)$, 
\[
\norm{M}_{\mathrm{HS}}\eqdef \sqrt{\tr(MM^{\star})}<+\infty
\] 
is its Hilbert-Schmidt norm (in the finite-dimensional case is equivalent to the Frobenius norm). {We denote by $\wlim$ (resp. $\slim$) the limit for the weak (resp. strong) topology of $\H$}. The notation $A: \H\rightrightarrows \H$ means that $A$ is a set-valued operator from $\H$ to $\H$. For $f:\H\rightarrow\R$, the sublevel of $f$ at height $r\in\R$ is denoted $[f\leq r]\eqdef \{x\in{\H}: f(x)\leq r\}$. For $1 \leq p \leq +\infty$, $\Lp^p([a,b])$ is the space of measurable functions $g:\R\rightarrow\R$ such that $\int_a^b|g(t)|^p dt<+\infty$, with the usual adaptation when $p = +\infty$. On the probability space $(\Omega,\calF,\PP)$, $\Lp^p(\Omega;\H)$ denotes the (Bochner) space of $\H$-valued random variables whose $p$-th moment (with respect to the measure $\PP$) is finite. Other notations will be explained when they first appear.

Let us recall some important definitions and results from convex analysis; for a comprehensive coverage, we refer the reader to \cite{rocka}.

\smallskip

We denote by $\Gamma_0(\H)$ the class of proper lsc and convex functions on $\H$ taking values in $\Rinf$.
For $\mu > 0$, $\Gamma_{\mu}(\H) \subset \Gamma_0(\H)$ is the class of $\mu$-strongly convex functions, \ie, functions $f$ such that $f-\frac{\mu}{2}\Vert\cdot\Vert^2$ is convex. We denote by $C^{s}(\H)$ the class of $s$-times continuously differentiable functions on $\H$. 
For $L \geq 0$, $C_L^{1,1}(\H) \subset C^{1}(\H)$ is the set of functions on $\H$ whose gradient is $L$-Lipschitz continuous, {and $C_L^2(\H)$ is the subset of $C_L^{1,1}(\H)$ whose functions are twice differentiable}.

The \textit{subdifferential} of a function $f\in\Gamma_0(\H)$ is the set-valued operator $\partial f:\H \rightrightarrows \H$ such that, for every $x$ in $\H$,
\[
\partial f(x)=\{u\in\H:f(y)\geq f(x) + \dotp{u}{y-x} \qforallq y\in\H\},
\]
which is non-empty for every point in the relative interior of the domain of $f$. When $f$ is finite-valued, then $f$ is continuous, and $\partial f(x)$ is a non-empty convex and compact set for every $x\in \H$. If $f$ is differentiable, then $\partial f(x)=\{\nabla f(x)\}$. For every $x\in \H$ such that $\partial f(x) \neq \emptyset$, the minimum norm selection of $\partial f(x)$ is the unique element $\{\partial^0 f(x)\} \eqdef \argmin_{u\in \partial f(x)}\norm{u}$. The projection of a point $x\in \H$ onto a non-empty closed convex set $C\subseteq \H$ is denoted by $\proj_C(x)$.

\subsection{\texorpdfstring{\tcm{Assumptions on volatility and damping parameters}}{Assumptions on volatility and damping parameters}}
Recall that our focus in this paper is on an optimization perspective, and as we argued in the introduction, we will study the long time behavior of our SDEs and SDIs (in particular \eqref{ISIHD-S} and \eqref{ISIHDN-S}) as the diffusion term vanishes when $t \to +\infty$. Therefore, throughout the paper, we assume that the diffusion (volatility) term $\sigma$ satisfies:
\begin{equation*}
\tag{$\mathrm{H}_{\sigma}$}\label{H}
\begin{cases}
\sup_{t \geq t_0,x \in \H} \Vert\sigma(t,x)\Vert_{\HS}<+\infty, \\
\Vert\sigma(t,x')-\sigma(t,x)\Vert_{\HS}\leq l_0\norm{x'-x},
\end{cases}
\end{equation*}   
for some $l_0>0$ and for all  $t\geq t_0, x, x'\in \H$. The Lipschitz continuity assumption is mild and classical and will be required to ensure the well-posedness of \eqref{ISIHD-S} and \eqref{ISIHDN-S}. Let us also define $\sigma_{\infty}:[t_0,+\infty[\rightarrow \R_+$ as
\[
\sigma_{\infty}(t)\eqdef \sup_{x\in\H}\Vert\sigma(t,x)\Vert_{\HS}.
\]
\begin{remark}\label{sigma*}
\eqref{H} implies the existence of $\sigma_*>0$ such that: $$\Vert\sigma(t,x)\Vert_{\HS}^2=\tr[\Sigma(t,x)]\leq \sigma_*^2,$$ 
for all $t\geq t_0, x\in \H$, where $\Sigma\eqdef \sigma\sigma^{\star}$.
\end{remark} 
For $t_0> 0$, let $\gamma:[t_0,+\infty[\rightarrow \R_+$ be a viscous damping and define: 
\begin{equation*}
p(t)\eqdef \exp\left(\int_{t_0}^t \gamma(s)ds\right), \quad  \Gamma(t)\eqdef p(t)\int_t^{\infty}\frac{ds}{p(s)} 
\qandq I[h](t) \eqdef \exp\left(-\int_{t_0}^t\frac{du}{\Gamma(u)}\right)\int_{t_0}^t h(u)\frac{\exp\left(\int_{t_0}^u\frac{ds}{\Gamma(s)}\right)}{\Gamma(u)}du.
\end{equation*}
We assume that
\begin{align}\tag{$\mathrm{H}_\gamma$}\label{H1}
\begin{cases}
    \text{$\gamma$ is upper bounded by a non-increasing function for every $t\geq t_0$;}\\
     \int_{t_0}^{\infty}\frac{ds}{p(s)}<+\infty.
\end{cases}
\end{align}

\begin{remark}\label{rem:Gamma}
Let us notice that $\Gamma$ satisfies the relation $\Gamma'=\gamma\Gamma-1.$
\end{remark}


\tcm{
\subsection{Reminder of main previous results} 
Before we delve into our core contributions, it is important to note that we will require some specific results gleaned from \cite{mio,mio2}. These are the subject of the following paragraphs.
}

\paragraph{Results on \eqref{CSGDintro}.} 
Since we are going to show results in the smooth case, we rewrite \eqref{H0} when $g \equiv 0$, 
\begin{align}\label{H01}
\begin{cases}
\text{$f:\H\rightarrow\R$ is continuously differentiable and convex with $L$-Lipschitz continuous gradient}; \\
\calS\eqdef \argmin (f)\neq\emptyset. \tag{$\mathrm{H}_0^{'}$} 
\end{cases}
\end{align}

\begin{theorem}[{\cite[Theorem~3.1]{mio}}]\label{3.1}
Consider $f$ and $\sigma$ that satisfy Assumptions~\eqref{H01} and \eqref{H}. Let $\nu\geq 2$ and consider the SDE:
\begin{equation}\label{a1}
\begin{cases}
\begin{aligned}
    dX(t)&= -\nabla f(X(t))dt+\sigma(t,X(t))dW(t),\\
    X(t_0)&=X_0,  
\end{aligned}
\end{cases}
\end{equation}
where $X_0\in \Lp^{\nu}(\Omega;\H)$. Then, there exists a unique solution $X\in S_{\H}^{\nu}[t_0]$ (see Section~\ref{onstochastic} for the notation) of \eqref{a1}.  Additionally, if $\sigma_{\infty}\in \Lp^2([t_0,+\infty[)$, then: 
\begin{enumerate}[label=(\roman*)]
\item \label{acota} $\sup_{t\geq 0}\EE[\norm{ X(t)}^2]<+\infty$.
\vspace{1mm}
\item  $\forall x^{\star}\in \calS$, $\lim_{t\rightarrow +\infty} \norm{ X(t)-x^{\star}}$ exists a.s. and $\sup_{t\geq 0}\norm{ X(t)}<+ \infty$ a.s..
\vspace{1mm}
\item \label{iiconv} 
$\lim_{t\rightarrow \infty}\norm{\nabla f(X(t))}=0$ a.s..  
As a result, $\lim_{t\rightarrow \infty} f(X(t))=\min f$ a.s..
\vspace{1mm}
\item  There exists an $\calS$-valued random variable $X^{\star}$ such that $\wlim_{t\rightarrow +\infty} X(t) = X^{\star}$ a.s..
 \end{enumerate}
\end{theorem}
 
\begin{theorem}[{\cite[Theorem~3.4]{mio}}]\label{3.4}
Let $\nu\geq 2$ and consider the SDE \eqref{a1} with initial data $X_0\in\Lp^{\nu}(\Omega;\H)$, where $f$ and $\sigma$ satisfy Assumptions \eqref{H0} and \eqref{H}. Moreover, we assume that $\sigma$ satisfies  $t\mapsto t\sigma_{\infty}^2(t)\in \Lp^1([t_0,+\infty[)$ and that {either $\H$ is finite-dimensional or $f\in C^2(\H)$}. Then, the solution trajectory $X \in S_{\H}^{\nu}[t_0]$ is unique and we have that:
\begin{enumerate}[label=(\roman*)]
\item $\EE[f(X(t))-\min f]=\mathcal{O}(t^{-1}).$
\end{enumerate}
 Moreover, if $f\in C^2(\H)$, then the following hold:
\begin{enumerate}[label=(\roman*),resume]
    \item $t\mapsto t\Vert \nabla f(X(t))\Vert^2\in \Lp^1([t_0,+\infty[)$ a.s..
    \item $f(X(t))-\min f=o(t^{-1})$ a.s..
\end{enumerate}
\end{theorem} 

\paragraph{Results on \eqref{SDI}} 
The far more intricate non-smooth version of \eqref{CSGDintro} has been recently studied in \cite{mio2}. Let $F,\sigma$ satisfy \eqref{H0} and \eqref{H}. We consider the stochastic differential inclusion 
\begin{equation}\label{SDI}\tag{$\mathrm{SGI}$}
\begin{cases}
\begin{aligned}
dX(t)&\in -\partial F(X(t))dt+\sigma(t,X(t))dW(t), \quad t> t_0,\\
X(t_0)&=X_0.
\end{aligned}
\end{cases}
\end{equation}

The following definition makes precise the notion of solution that we are interested in.
\begin{definition}\label{def:SDIsolution}
A solution of \eqref{SDI} is  a couple $(X,\eta)$ of $\calF_t$-adapted processes such that almost surely:
\begin{enumerate}[label=(\roman*)]
    \item $X$ is continuous with sample paths in the domain of $\partial g$.
    \item $\eta:[t_0,+\infty[\rightarrow \H$ is absolutely continuous, such that $\eta(t_0)=0$, and $\forall T>t_0$, $\eta'\in\Lp^2([t_0,T];\H)$,    $\eta'(t)\in \partial g(X(t))$ for almost all $t\geq t_0$;
    \item For $t> t_0$,
   \begin{equation}\label{itosdis}
 \begin{aligned}
 \begin{cases}
     X(t)&=X_0-\int_{t_0}^t \nabla f(X(s))ds-\eta(t)+\int_{t_0}^t \sigma(s,X(s))dW(s),\\
      X(t_0)&=X_0.
      \end{cases}
 \end{aligned}
     \end{equation}
\end{enumerate}
\end{definition}

For brevity, we sometimes omit the process $\eta$ and say that $X$ is a solution of \eqref{SDI}, meaning that,  there exists a process $\eta$ such that $(X,\eta)$ satisfies the previous definition. The definition of uniqueness for the process $X$ is given in Section~\ref{onstochastic}.

In order to show the main results for \eqref{SDI}, we consider the sequence of solutions $(X_{\lambda})_{\lambda>0}$ of the SDEs
\begin{equation}\label{SDEL}\tag{$\mathrm{SDE_{\lambda}}$}
\begin{cases}
\begin{aligned}
dX_{\lambda}(t)&= -\nabla (f+g_{\lambda})(X_{\lambda}(t))dt+\sigma(t,X_{\lambda}(t))dW(t), \quad t> t_0,\\
X_{\lambda}(t_0)&=X_0,
\end{aligned}
\end{cases}
\end{equation}
where 
$g_{\lambda}$ is the Moreau envelope of $g$ with parameter $\lambda>0$. \tcm{Observe that system \eqref{SDEL} was also directly studied in \cite{mio} to solve \eqref{P}, where a complete discussion is provided. It is worth emphasizing that the Moreau envelope of $g$, and the corresponding system \eqref{SDEL} is only needed as a means to establish existence and uniqueness of solution of \eqref{SDI}. This follows the same reasoning as in the deterministic case where the Moreau-Yosida regularization and nonlinear semigroup theory have proven very useful to show existence and uniqueness for differential inclusions \cite{brezis}. However, in the convergence analysis, we never replace $g$ by its Moreau envelope.} 

We define the integrability condition that for every $T>t_0$,
\begin{equation}\label{Hl}\tag{$\mathrm{H}_{\lambda}$}
\limsup_{\lambda\downarrow 0}\int_{t_0}^T \EE(\Vert \nabla g_{\lambda}(X_{\lambda}(t))\Vert^2)dt<+\infty, 
\end{equation} 
\begin{remark}
Some conditions on $g$ for \eqref{Hl} to be satisfied are discussed in \cite{Pettersson95}. Of particular interest, which covers a broad class of functions, if when $\partial g$ has full domain and there exists $C_0>0$ such that: 
\[
\norm{\partial^0 g(x)}\leq C_0(1+\Vert x\Vert),\quad \forall x\in\H.
\]
This is for instance the case when $g$ is Lipschitz continuous.
\end{remark}

The following results on \eqref{SDI} were proved in \cite{mio2}. 
\begin{theorem}[\cite{mio2}]\label{converge}
Consider $F=f+g$ and $\sigma$ satisfying \eqref{H0} and \eqref{H}. {Suppose further that $g$ verifies \eqref{Hl}}. 
Let $\nu\geq 2$, $t_0\geq 0$ , and consider the dynamic \eqref{SDI} with initial data $X_0\in \Lp^{\nu}(\Omega;\H)$.
Then, \eqref{SDI} has a unique solution in the sense of Definition~\ref{def:SDIsolution} $(X,\eta)\in S_{\H}^{\nu}[t_0]\times C^{1}([t_0,+\infty[;\H)$.

Moreover, if $\sigma_{\infty}\in\Lp^2([t_0,+\infty[)$, then the following holds:
    \begin{enumerate}[label=(\roman*)]
        \item \label{acota2fg} $\EE[\sup_{t\geq t_0}\norm{X(t)}^{\nu}]<+\infty$.
        \vspace{1mm}
        \item \label{bou} $\forall x^{\star}\in \calS_F$, $\lim_{t\rightarrow +\infty} \norm{X(t)-x^{\star}}$ exists a.s. and $\sup_{t\geq t_0}\norm{ X(t)}<+ \infty$ a.s..
        \vspace{1mm}
        \item \label{iiconvfg} If $g$ is continuous 
        , then $\nabla f$ is constant on $\calS_F$, there exists $x^{\star}\in \calS_F$ such that  $\slim_{t\rightarrow \infty}\nabla f(X(t))=\nabla f(x^{\star})$ a.s., and  
        \[
        \int_{t_0}^{+\infty} F(X(t))-\min F \hspace{0.1cm}dt<+\infty \quad a.s..
        \]
        \item There exists an $\calS_F$-valued random variable $X^{\star}$ such that $\wlim_{t\rightarrow +\infty} X(t)= X^{\star}$.
    \end{enumerate}

\end{theorem}

\paragraph{Tikhonov regularization.} Let us now turn to a Tikhonov regularization of \eqref{SDI}, \ie,
\begin{equation}\label{SDITA}\tag{$\mathrm{SGI-TA}$}
\begin{cases}
\begin{aligned}
dX(t)&\in-\partial F(X(t)) -\varepsilon(t) X(t) +\sigma(t,X(t))dW(t), \quad t\geq t_0,\\
X(t_0)&=X_0 .
\end{aligned}
\end{cases}
\end{equation}
Solution existence and uniqueness for \eqref{SDITA} is established in \cite[Theorem~3.3]{mio2}. We also have the following.

\begin{theorem}[{\cite[Theorem~4.1]{mio2}}]\label{converge20}
Let $\nu\geq 2$ and consider the dynamic \eqref{SDITA} with initial data $X_0\in \Lp^{\nu}(\Omega;\H)$, where $F=f+g$ and $\sigma$ satisfy Assumptions \eqref{H0} and \eqref{H}. Furthermore, assume that $g$ satisfies \eqref{Hl}. Then, there exists a unique solution $X\in S_{\H}^{\nu}[t_0]$ of \eqref{SDITA}. Let $x^{\star}=\proj_{\calS_F}(0)$ be the minimum norm solution, and for $\varepsilon>0$ let $x_{\varepsilon}$ be the unique minimizer of $F_{\varepsilon}(x)\eqdef F(x)+\frac{\varepsilon}{2}\|x\|^2$. Suppose that $\sigma_{\infty}\in \Lp^2([t_0,+\infty[)$, and that $\varepsilon:[t_0,+\infty[\rightarrow\R_+$ satisfies the conditions:

\vspace{2mm}
\begin{enumerate}[label=\rm{$(T_\arabic*)$}]
    \item \label{t1} $\varepsilon (t) \to 0$ \textit{as} $t \to +\infty$;
    \item \label{t2} $\displaystyle{\int_{t_0}^{+\infty} \varepsilon (t) dt = +\infty}$;
    \item \label{t3} $\displaystyle{ \int_{t_0}^{+\infty} \varepsilon (t) \left(    \|x^{\star} \|^2 -  \|x_{\varepsilon (t)} \|^2  \right)dt  <+\infty}. $
\end{enumerate}

\vspace{2mm}

\noindent Then, we have

\begin{enumerate}[label=(\roman*)]

\item   $\sup_{t\geq t_0}\norm{ X(t)}<+ \infty$ a.s., and

\item $\slim_{t\rightarrow +\infty} X(t)=x^{\star}$ a.s..
\vspace{1mm}
 \end{enumerate}
\end{theorem}

This means that we can obtain almost sure strong convergence of the trajectory to a particular (non-random) solution: the one of minimal norm.

\section{From first-order to second-order systems}\label{sec:scaling}
\subsection{Time scaling and averaging}
We apply a time scaling and then an averaging technique to the system \eqref{SDI} to derive an insightful reparametrization of a particular case of our second-order system \eqref{ISIHDN-S}, specifically, the case when $\beta\equiv \Gamma$. The main advantage of this method is that the results of \eqref{SDI} directly carry over to obtain results on the convergence behavior of \eqref{ISIHDN-S} without passing through a dedicated Lyapunov analysis. \tcm{However, as discussed in the introduction, the averaging technique is restricted to convex objectives, as it heavily relies on Jensen's inequality.}

\smallskip

Let $\nu\geq 2$, $s_0 > 0$. We consider the potential $F=f+g$ where $g$ satisfies \eqref{Hl}. Let $\sigma_1$ be a diffusion term in the time parametrization by $s$. We will study the dynamic \eqref{SDI} in $s$, starting at $s_0$, with diffusion term $\sigma_1$ under hypotheses \eqref{H0} and \eqref{H}. Let $\sigma_{1_*}>0$ be such that $$\Vert\sigma_1(s,x)\Vert_{\HS}\leq\sigma_{1_*}^2,\quad \forall s\geq s_0,\forall x\in\H,$$ and $\sigma_{1_{\infty}}(s)\eqdef \sup_{x\in\H}\Vert\sigma_1(s,x)\Vert_{\HS}$. We rewrite \eqref{SDI} in a format adapted to our case,
\begin{equation}\label{CSGD}
\begin{cases}
\begin{aligned}
d{Z(s)}&\in-\partial F({Z(s)})ds+\sigma_1(s,{Z(s)})dW(s), \quad s> s_0,\\
{Z(s_0)}&={Z_0} ,
\end{aligned}
\end{cases}
\end{equation}
where ${Z_0}\in\Lp^{\nu}([s_0,+\infty[;\H)$.

Let us make the change of time $s=\tau(t)$ in the dynamic \eqref{CSGD}, where $\tau$ is an increasing function from $[t_0,+\infty[$ to $[s_0,+\infty[$, which is twice differentiable, and which satisfies $\lim_{t\rightarrow+\infty}\tau(t)=+\infty$. Denote ${Y(t)}\eqdef {Z}(s)$ and $t_0$ be such that $s_0=\tau(t_0)$. By the chain rule and \cite[Theorem~8.5.7]{oksendal_2003}, we have 
\begin{equation}\label{scaling}
\begin{cases}
\begin{aligned}
d{Y(t)}&\in -\tau'(t)\partial F({Y(t)})dt+\sqrt{\tau'(t)}\sigma_1(\tau(t), {Y(t)})dW(t), \quad t>t_0,\\
{Y(t_0)}&={Z_0}.
\end{aligned}
\end{cases}
\end{equation}

Consider the smooth case, i.e. when $g\equiv 0$ and the hypotheses of Theorem~\ref{3.4} ($f\in C_L^2(\H)$ and $\sigma_1\in\Lp^2([s_0,+\infty[)$), then we can conclude that the convergence rate of \eqref{scaling} (when $g\equiv 0$) is the following 
\begin{equation}
    f({Y(t)})-\min f=o\left(\frac{1}{\tau(t)}\right) \text{ a.s.}.
\end{equation}

By introducing a function $\tau$ that grows faster than the identity ($\tau(t)\geq t$), we have accelerated the dynamic, passing from the asymptotic convergence rate $1/s$ for \eqref{CSGD} to $1/\tau(t)$ for \eqref{scaling}. The price to pay is {that the drift term in \eqref{scaling} is non-autonomous}, furthermore, when the coefficient in front of the gradient tends to infinity as $t\rightarrow+\infty$, it will preclude the use of an explicit discretization in time. To overcome this, we adapt from \cite{fast} the following approach, which is called averaging. \smallskip

Consider \eqref{scaling} and define the two stochastic processes ${X},{V}:\Omega\times[t_0,+\infty[\rightarrow\H$ as
\begin{equation}\label{eq:XYVrel}
\begin{cases}
d{X(t)}&={V(t)}dt, \quad t>t_0,\\
{Y(t)}& = {X(t)}+\tau'(t){V(t)}, \quad t>t_0,\\
{X(t_0)}& ={X_0},\quad {V(t_0)}= {V_0},
\end{cases}
\end{equation}
where $Y(t)$ is the process in \eqref{scaling}, and ${X_0},{V_0}\in \Lp^{\nu}(\Omega;\H)$ are initial data. This leads us to set ${Z_0}\eqdef {X_0}+\tau'(t_0){V_0}$ in order for the equations to fit.
According to the averaging, the differential form of ${Y(t)}$ is $$d{Y(t)}=d{X(t)}+\tau''(t){V(t)}dt+\tau'(t)d{V(t)}.$$ 
Combining the previous equation with \eqref{scaling}, we have that $$-\tau'(t)\partial F({Y(t)})dt+\sqrt{\tau'(t)}\sigma_1(\tau(t), {Y(t)})dW(t)\ni d{X(t)}+\tau''(t){V(t)}dt+\tau'(t)d{V(t)}.$$
Using that $d{X(t)}={V(t)}dt$ and dividing by $\tau'$, we then have $$-\partial F({X(t)}+\tau'(t){V(t)})dt+\frac{1}{\sqrt{\tau'(t)}}\sigma_1(\tau(t), {X(t)}+\tau'(t){V(t)})dW(t)\ni\frac{1+\tau''(t)}{\tau'(t)}{V(t)}dt+d{V(t)}.$$
Therefore, after the time scaling and averaging, we obtain the following dynamic:
\[
\begin{cases}\label{isihd-s1}\tag{ISIHD-S.1}
d{X(t)}&={V(t)}dt, \quad t>t_0,\\
d{V(t)}&\in -\frac{1+\tau''(t)}{\tau'(t)}{V(t)}dt-\partial F({X(t)}+\tau'(t){V(t)})dt \\
&\phantom{=}+\frac{1}{\sqrt{\tau'(t)}}\sigma_1(\tau(t),{X(t)}+\tau'(t){V(t)})dW(t),\quad t>t_0, \\
{X(t_0)}&={X_0}, \quad {V(t_0)}={V_0}.
\end{cases}
\]
Let $\gamma:[t_0,+\infty[\rightarrow\R_+$ satisfying \eqref{H1}. {We are going to determine $\tau$ in order to obtain a viscous damping coefficient equal to $\gamma$}, i.e., 
\begin{align*}
    \frac{1+\tau''(t)}{\tau'(t)}&=\gamma(t).
\end{align*}
Clearly, $\tau'$ solves the following ODE in $\zeta$ 
\[
\zeta'=\gamma \zeta-1.
\]
As observed in Remark~\ref{rem:Gamma}, the function $\Gamma$ also satisfies the same ODE, and thus we can adjust the initial condition of $\tau'$ to obtain 
\[
\tau'(t)=\Gamma(t)=p(t)\int_{t}^{\infty}\frac{du}{p(u)} \quad \forall t \geq t_0 .
\]
{We then integrate and take $\tau(t)=s_0+\int_{t_0}^t \Gamma(u)du$ to get $\tau(t_0)=s_0$ as required. This is a valid selection of $\tau$ since $t\mapsto s_0+\int_{t_0}^t \Gamma(u)du$ is an increasing function from $[t_0,+\infty[$ to $[s_0,+\infty[$, twice differentiable and $\Gamma\notin\Lp^1([t_0,+\infty[)$ because $\Gamma$ is lower bounded by a non-decreasing function since $\gamma$ is upper bounded by a non-increasing function (see \cite[Proposition 2.2]{cabot}) by \eqref{H1}. For this particular selection of $\tau$, and defining} $\tilde{\sigma}_1(t,\cdot)\eqdef\frac{\sigma_1(\tau(t),\cdot)}{\sqrt{\Gamma(t)}}$, we have that \eqref{isihd-s1} is equivalent to
\[
\begin{cases}\label{isihd-s2}\tag{ISIHD-S.2}
d{X(t)}&={V(t)}dt,\quad t>t_0,\\
d{V(t)}&\in -\gamma(t){V(t)}dt-\partial F({X(t)}+\Gamma(t){V(t)})dt+\tilde{\sigma}_1(t, {X(t)}+\Gamma(t){V(t)})dW(t), \quad t>t_0,\\
{X(t_0)}&={X_0},\quad {V(t_0)}={V_0}.
\end{cases}
\]
{Clearly, \eqref{isihd-s2} is nothing but \eqref{ISIHDN-S}} when $\beta\equiv\Gamma$ and $\sigma\equiv\tilde{\sigma}_1$.

\medskip

In order to be able to transfer the convergence results on $Z$ in \eqref{CSGD} (via \eqref{scaling}) to $X$ in \eqref{isihd-s2}, it remains to express ${X}$ in terms of ${Y}$ only. For this, let 
\[
a(t)\eqdef\frac{1}{\tau'(t)}, \quad A(t)\eqdef\int_{t_0}^t a(u)du.
\] 
Recalling the averaging in \eqref{eq:XYVrel}, we need to integrate the following equation
\begin{equation}
{V(t)}+a(t){X(t)}=a(t){Y(t)}.
\end{equation}
Multiplying both sides by $e^{A(t)}$ and using \eqref{eq:XYVrel}, we get equivalently 
\begin{equation}\label{initia}
d\left(e^{A(t)}{X(t)}\right)=a(t)e^{A(t)}{Y(t)} dt.
\end{equation}
Integrating and using again \eqref{eq:XYVrel}, we obtain 
\begin{align*}
{X(t)}&=e^{-A(t)}{X(t_0)}+e^{-A(t)}\int_{t_0}^t a(u)e^{A(u)}{Y(u)}du\\
    &=e^{-A(t)}{Y(t_0)}+e^{-A(t)}\int_{t_0}^t a(u)e^{A(u)}{Y(u)}du-e^{-A(t)}\tau'(t_0){V(t_0)}.
\end{align*}
Then we can write 
\begin{equation}\label{aver}
    {X(t)}=\int_{t_0}^t {Y(u)}d\mu_t(u)+\xi(t),
\end{equation}
where $\mu_t$ is the probability measure on $[t_0,t]$ defined by 
\begin{equation}\label{medida}
    \mu_t\eqdef e^{-A(t)}\delta_{t_0}+a(u)e^{A(u)-A(t)}du,
\end{equation}
where $\delta_{t_0}$ is the Dirac measure at $t_0$, $a(u)e^{A(u)-A(t)}du$ is the measure with density $a(\cdot)e^{A(\cdot)-A(t)}$ with respect to the Lebesgue measure on $[t_0,t]$, and $\xi(t)$ is a random process since $V_0$ is a random variable, i.e.,
\begin{equation}\label{chi}
\xi(t)\eqdef \xi(\omega,t)=-e^{-A(t)}\tau'(t_0){V_0}(\omega) \quad  \forall \omega \in \Omega . 
\end{equation}


\subsection{Convergence of the trajectory and convergence rates under general \texorpdfstring{$\gamma$}{}, and \texorpdfstring{$\beta\equiv\Gamma$}{}}
We here state the main results of this section. We first show almost sure convergence of the trajectory of \eqref{ISIHDN-S} to a random variable taking values in the set of minimizers of $F$. When $g \equiv 0$, we also provide convergence rates.

\smallskip

\tcm{The following result imposes an integrability condition on the diffusion term, which is essential for ensuring the almost sure weak convergence of the trajectory. More precisely, this integrability condition allows to show that the trajectory asymptotically converges almost surely, in the weak topology, to a random variable that takes values in the set of minimizers of $F$.} 

\begin{theorem}\label{trajectory}
 Let $\nu\geq2$ and consider the dynamic \eqref{ISIHDN-S} with initial data ${X_0},{V_0}\in\Lp^{\nu}(\Omega;\H)$, where $\gamma:[t_0,+\infty[\rightarrow\R_+$ satisfies \eqref{H1}, and $\beta\equiv\Gamma$. Besides, $F=f+g$ and $\sigma$ satisfy Assumptions \eqref{H0} and \eqref{H}. Moreover, suppose that $g$ satisfies \eqref{Hl}. Then, there exists a unique solution $({X,V})\in S_{\H\times \H}^{\nu}[t_0]$ of \eqref{ISIHDN-S}.  Additionally, if $\Gamma\sigma_{\infty}\in \Lp^2([t_0,+\infty[)$, then there exists an $\calS_F$-valued random variable ${X^{\star}}$ such that $\wlim_{t\rightarrow\infty} {X(t)} = {X^{\star}}$ a.s. and $\wlim_{t\rightarrow\infty}\Gamma(t){V(t)}=0.$ a.s..
\end{theorem}

\proof{}
Let $\theta(t)\eqdef \int_{t_0}^t \Gamma(u)du$, $\tilde{\sigma}(s,\cdot)\eqdef\sigma(\theta^{-1}(s),\cdot)\sqrt{\Gamma(\theta^{-1}(s))}$, and $\tilde{\sigma}_{\infty}(s)\eqdef \sup_{x\in\H} \Vert \tilde{\sigma}(s,x)\Vert_{\HS} $. Then $\Gamma\sigma_{\infty}\in \Lp^2([t_0,+\infty[)$ is equivalent to $\tilde{\sigma}_{\infty}\in\Lp^2(\R_+)$.
Consider the dynamic: \smallskip

\begin{equation}
\begin{cases}\label{inv}
d{Z(s)}&\in -\partial F({Z(s)})+\tilde{\sigma}(s,{Z(s)})dW(s),\quad s>0,\\
{Z(0)}&={X_0}+\Gamma(t_0){V_0}.
\end{cases}    
\end{equation}

By Theorem~\ref{converge}, we have that there exists a unique solution $({Z},\eta)\in S_{\H}^{\nu}\times C^1(\R_+;\H)$ of \eqref{inv}, and an $\calS_F$-valued random variable ${X^{\star}}$ such that $\wlim_{s\rightarrow\infty} {Z(s)} =  {X^{\star}}$ a.s.. Moreover, using the time scaling $\tau\equiv\theta$ and the averaging described in this section, we end up with the dynamic \eqref{ISIHDN-S} in the case where $\beta\equiv\Gamma$.\smallskip

 It is direct to check that the time scaling and averaging preserves the uniqueness of a solution \linebreak
 $({X,V})\in~S_{\H\times\H}^0[t_0]$. Now let us validate $({X,V})\in S_{\H\times\H}^{\nu}[t_0]$. Since $$\EE\left(\sup_{s\in[0,T]}\Vert {Z(s)}\Vert^{\nu}\right)<+\infty, \quad \forall T>0,$$
 we directly obtain $$\EE\left(\sup_{t\in[t_0,T]}\Vert {Y(t)}\Vert^{\nu}\right)<+\infty,\quad \forall T>t_0.$$
 Thanks to the relation \eqref{aver}, 
the following holds
 \begin{align*}
     \Vert {X(t)}\Vert^{\nu}&\leq \nu\left(\Big\Vert {X(t)}-\int_{t_0}^t {Y(u)}d\mu_t(u)\Big\Vert^{\nu}+\Big\Vert\int_{t_0}^t {Y(u)}d\mu_t(u)\Big\Vert^{\nu}\right)\\
     &\leq \nu\left(\Vert \xi(t)\Vert^{\nu}+(t-t_0)^{\nu-1}\int_{t_0}^t \Vert {Y(u)}\Vert^{\nu}d\mu_t(u)\right).
 \end{align*}
Let $T>t_0$ be arbitrary. Taking supremum over $[t_0,T]$ and then expectation at both sides, we obtain that 
\begin{align*}
     \EE\left(\sup_{t\in [ t_0,T]}\Vert {X(t)}\Vert^{\nu}\right)\leq \nu\left(\EE(\Vert {V_0}\Vert^{\nu})\Vert\Gamma(t_0)\Vert^{\nu}+(T-t_0)^{\nu-1}\EE\left(\sup_{t\in [ t_0,T]}\Vert {Y(t)}\Vert^{\nu}\right)\right)<+\infty.
 \end{align*}
 Since ${V(t)}=\frac{{Y(t)}-{X(t)}}{\Gamma(t)}$, we have \begin{align*}
     \Vert {V(t)}\Vert^{\nu}\leq \frac{\nu}{\Gamma^{\nu}(t)}(\Vert {Y(t)}\Vert^{\nu}+\Vert {X(t)}\Vert^{\nu}).
 \end{align*}
Similarly as before, we let $T>t_0$ be arbitrary, and take the supremum over $[t_0,T]$ and then expectation at both sides to obtain \begin{align*}
    \EE\left(\sup_{t\in [ t_0,T]}\Vert {V(t)}\Vert^{\nu}\right)\leq \nu\sup_{t\in[t_0,T]}\frac{1}{\Gamma^{\nu}(t)}\left(\EE\left(\sup_{t\in [ t_0,T]}(\Vert {Y(t)}\Vert^{\nu}+\Vert {X(t)}\Vert^{\nu})\right)\right).
\end{align*}
Since $\Gamma$ is a continuous positive function, by the extreme value theorem, we have that there exists $t_T\in [t_0,T]$ such that $\sup_{t\in[t_0,T]}\frac{1}{\Gamma^{\nu}(t)}= \frac{1}{\Gamma^{\nu}(t_T)}<+\infty$, and we conclude that $({X,V})\in S_{\H\times\H}^{\nu}[t_0]$.\smallskip
 
Now we prove that there exists an $\calS_F-$valued random variable ${X^{\star}}$ such that $\wlim_{t\rightarrow\infty} {X(t)} =  {X^{\star}}$ a.s.. By virtue of Theorem~\ref{converge}, there exists an $\calS_F-$valued random variable ${X^{\star}}$ such that $\wlim_{s\rightarrow\infty} {Z(s)} =  {X^{\star}}$ a.s.. We also notice that we have directly $\wlim_{t\rightarrow\infty} {Y(t)} = {X^{\star}}$ a.s.. Let $h\in\H$ be arbitrary and use the relation \eqref{aver} as follows: \begin{align*}
    \vert \langle {X(t)}-{X^{\star}},h\rangle\vert&\leq \bigg|\bigg\langle {X(t)}-\int_{t_0}^t {Y(u)}d\mu_t(u),h\bigg\rangle\bigg|+\bigg|\bigg\langle\int_{t_0}^t {Y(u)}d\mu_t(u)-{X^{\star}},h\bigg\rangle\bigg|\\
    &= \vert\langle \xi(t),h\rangle\vert+\bigg|\bigg\langle\int_{t_0}^t ({Y(u)}-{X^{\star}})d\mu_t(u),h\bigg\rangle\bigg|\\
    &= \vert\langle \xi(t),h\rangle\vert+\bigg|\int_{t_0}^t\langle {Y(u)}-{X^{\star}},h\rangle d\mu_t(u)\bigg|\\
    &\leq \Vert\xi(t)\Vert\Vert h\Vert+\int_{t_0}^t\vert\langle  {Y(u)}-{X^{\star}},h\rangle \vert d\mu_t(u),
\end{align*}
where the second equality comes from the dominated convergence theorem, since $\sup_{s>t_0}\Vert {Y(s)}\Vert<+\infty$ a.s. (by~\ref{bou} of Theorem~\ref{converge}).

\noindent Now let $a(t)=\frac{1}{\Gamma(t)}$ and $A(t)=\int_{t_0}^t\frac{du}{\Gamma(u)}$. By Lemma~\ref{1gam} (defined in Section \ref{detaux}) and the fact that $V_0\in\Lp^{\nu}(\Omega;\H)$, we have that $\lim_{t\rightarrow\infty}\Vert\xi(t)\Vert=0$ a.s.. On the other hand, it holds that 
\begin{align*}
\int_{t_0}^t \vert\langle {Y(u)}-{X^{\star}},h\rangle\vert d\mu_t(u)&\leq e^{-A(t)}\vert\langle {Y(t_0)}-{X^{\star}},h\rangle\vert+e^{-A(t)}\int_{t_0}^t a(u)e^{A(u)}\vert \langle {Y(u)}-{X^{\star}},h\rangle\vert du.
\end{align*}
Now let $b(u)=\vert \langle {Y(u)}-{X^{\star}}, h\rangle\vert$. Since we already proved that $\lim_{u\rightarrow\infty}b(u)=0$ a.s., and we have that $a\notin \Lp^1([t_0,+\infty[)$ by Lemma~\ref{1gam}, we utilize Lemma~\ref{ab} (also defined in Section \ref{detaux}) with our respective $a,b$ functions. This let us conclude that $$ \lim_{t\rightarrow+\infty}\vert \langle {X(t)}-{X^{\star}},h\rangle\vert=0 \quad a.s..$$
Thus, $\wlim_{t\rightarrow\infty} {X(t)} = {X^{\star}}$ a.s.. Finally, since $${Y(t)}={X(t)}+\Gamma(t){V(t)},$$
and the fact that ${X}$ and ${Y}$ have (a.s.) the same limit, we conclude that 
\[
\wlim_{t\rightarrow\infty}\Gamma(t){V(t)}=0 \quad a.s..
\]
\endproof
\tcm{In the smooth case, we also have convergence rates on the objective value and the gradient. In particular, the following two theorems will provide general abstract convergence rates under the same integrability condition on the diffusion term. These results will be specialized to specific choice of the parameters in Section \ref{fastconver}.}
\begin{theorem}\label{maxai}
Let $\nu\geq 2$ and consider the dynamic \eqref{ISIHD-S} with initial data ${X_0},{V_0}\in \Lp^{\nu}(\Omega;\H)$, such that $f$ and $\sigma$ satisfy \eqref{H01} and \eqref{H}, and in the case where $\gamma$ satisfies \eqref{H1}, $\beta\equiv\Gamma$. Moreover, suppose that either $\H$ is finite dimensional or $f\in C^2(\H)$, and $$t\mapsto \sqrt{\theta(t)}\Gamma(t)\sigma_{\infty}(t) \in \Lp^2([t_0,+\infty[),$$ where $\theta(t)\eqdef\int_{t_0}^t \Gamma(u)du$. Then the solution trajectory $({X,V})\in S_{\H\times\H}^{\nu}[t_0]$ is unique and satisfies: 
\[
\mathbb{E}[f({X(t)})-\min f]=\mathcal{O}\left(\max\Big\{e^{-A(t)},I\left[\frac{1}{\theta}\right](t)\Big\}\right),\quad \forall t> t_0,
\]
where $A(t)\eqdef\int_{t_0}^t\frac{du}{\Gamma(u)}$ and we recall that $I[\frac{1}{\theta}](t)= e^{-A(t)}\int_{t_0}^t \frac{1}{\theta(u)}\frac{e^{A(u)}}{\Gamma(u)}du$.
\end{theorem}
From hypothesis \eqref{H1} we have that $\lim_{t\rightarrow+\infty} e^{-A(t)}=0$, and since $\Gamma\notin\Lp^1([t_0,+\infty[)$, we can use Lemma~\ref{ab} to check that $\lim_{t\rightarrow\infty}I\left[\frac{1}{\theta}\right](t)=0$.



\begin{remark}
\tcm{If $\H$ is finite-dimensional and $g$ is a $L_0-$Lipschitz function, we could obtain convergence rates for the nonsmooth setting \eqref{P} by considering \eqref{ISIHD-S} with potential $F_{\rho}\eqdef f+g_{\rho}$, where $g_{\rho}$ is the Moreau envelope of $g$ with parameter $\rho>0$. With the same hypotheses on the initial conditions, $\sigma,\gamma,\beta$ as in Theorem \ref{maxai} and proceeding as in \cite[Section 5]{mio}, we could obtain \[
\mathbb{E}[F({X_{\rho}(t)})-\min F]=\mathcal{O}\left(\max\Big\{e^{-A(t)},I\left[\frac{1}{\theta}\right](t)\Big\}\right)+\frac{L_0^2}{2}\rho,\quad \forall t> t_0,
\]
where $(X_{\rho},V_{\rho})\in S_{\H\times\H}^{\nu}[t_0]$ is the solution of \eqref{ISIHD-S} with potential $F_{\rho}$.
}
\end{remark}

\proof{}
We will utilize the averaging technique used in Theorem~\ref{trajectory} and Jensen's inequality. First, we have 
\[
\mathbb{E}(f({X(t)})-\min f)=\mathbb{E}\left(f({X(t)})-f\left(\int_{t_0}^t {Y(u)}d\mu_t(u)\right)\right)+\mathbb{E}\left(f\left(\int_{t_0}^t {Y(u)}d\mu_t(u)\right)-\min f\right).
\]
Let us recall that $\mathbb{E}(\sup_{s\geq 0}\Vert {Z(s)}\Vert)<+\infty$, which implies that $\mathbb{E}(\sup_{t\geq t_0}\Vert {X(t)}\Vert)<+\infty$. We bound the first term using the convexity of $f$ and Cauchy-Schwarz inequality, \ie, that for every $x,y\in\H$, $$f(x)-f(y)\leq \langle\nabla f(x),x-y\rangle\leq \Vert \nabla f(x)\Vert \Vert y-x\Vert,$$ we get that
 \begin{align*}
f({X(t)})-f\left(\int_{t_0}^t {Y(u)}d\mu_t(u)\right) &\leq \Vert\nabla f({X(t)})\Vert\Vert\xi(t)\Vert\\
    &\leq \Vert\xi(t)\Vert(L\Vert {X(t)}\Vert+\Vert\nabla f(0)\Vert)\\
    &\leq \Vert\xi(t)\Vert\left(L\sup_{t\geq t_0}\Vert {X(t)}\Vert+\Vert\nabla f(0)\Vert\right),
\end{align*}
and we conclude that 
\[
\mathbb{E}\left(f({X(t)})-f\left(\int_{t_0}^t {Y(u)}d\mu_t(u)\right)\right)=\mathcal{O}(e^{-A(t)}).
\]
For the second term, we use Jensen's inequality to obtain 
\begin{align*}
\mathbb{E}\left(f\left(\int_{t_0}^t {Y(u)}d\mu_t(u)\right)-\min f\right)&\leq \int_{t_0}^t \mathbb{E}[f({Y(u)})-\min f]d\mu_t(u)\\
 &\leq e^{-A(t)}\mathbb{E}[f({Y(t_0)})-\min f]+e^{-A(t)}\int_{t_0}^t \frac{e^{A(u)}}{\Gamma(u)}\mathbb{E}(f({Y(u)})-\min f)du.
 \end{align*}
Since $\sqrt{\theta}\Gamma\sigma_{\infty}\in\Lp^2([t_0,+\infty[)$ is equivalent to $s\mapsto s\tilde{\sigma}_{\infty}^2(s)\in\Lp^1(\R_+)$, by Theorem~\ref{3.4}, we have that there exists $C>0$ such that $\mathbb{E}(f({Z(s)})-\min f)\leq \frac{C}{s}$. Then, we have $\mathbb{E}(f({Y(t)})-\min f)\leq \frac{C}{\theta(t)}$. Hence, there exists $C_0> 0$ such that
\[
\mathbb{E}\left(f({X(t)})-\min f\right)\leq C_0e^{-A(t)}+CI\left[\frac{1}{\theta}\right](t).
\]
\endproof

\begin{theorem}\label{thetagamma}
Let $\nu\geq 2$ and consider the dynamic \eqref{ISIHD-S} with initial data ${X_0},{V_0}\in \Lp^{\nu}(\Omega;\H)$, such that $f$ and $\sigma$ satisfy \eqref{H01} and \eqref{H}, and in the case where $\gamma$ satisfies \eqref{H1}, $\beta\equiv\Gamma$. Moreover, suppose that $f\in C^2(\H)$ and $$t\mapsto\sqrt{\theta(t)}\Gamma(t)\sigma_{\infty}(t)\in \Lp^2([t_0,+\infty[),$$
where $\theta(t)\eqdef\int_{t_0}^t\Gamma(u)du$. Then the solution trajectory $({X,V})\in S_{\H\times\H}^{\nu}[t_0]$ is unique and satisfies
\begin{equation}
    \int_{t_0}^{\infty}\theta(u)\Gamma(u)\Vert\nabla f({X(u)}+\Gamma(u){V(u)})\Vert^2du<+\infty\quad a.s..
\end{equation}
\end{theorem}


\proof{}
Consider \eqref{inv} and the technique used in Theorem~\ref{trajectory}. We have that $t\mapsto\theta(t)\Gamma^2(t)\sigma_{\infty}^2(t)\in \Lp^1([t_0,+\infty[)$ is equivalent to $s\mapsto s\tilde{\sigma}_{\infty}^2(s)\in\Lp^1(\R_+)$. Therefore, we can use Theorem~\ref{3.4} to state that $$\int_{0}^{\infty} s\Vert\nabla f({Z(s)})\Vert^2ds<+\infty\quad a.s..$$
Using the time scaling $\tau\equiv\theta$ and making the change of variable $\theta(t)=s$ in the previous integral, we obtain $$\int_{t_0}^{\infty}\theta(t)\Gamma(t)\Vert\nabla f({Y(t)})\Vert^2dt<+\infty\quad a.s..$$
Recalling that in the averaging we impose that ${Y}={X}+\Gamma {V}$, we conclude.
\endproof

\subsection{Fast convergence under \texorpdfstring{ $\alpha>3, \gamma(t)=\frac{\alpha}{t}$}{} and \texorpdfstring{$\beta(t)=\frac{t}{\alpha-1}$}{}}\label{fastconver}
In the following, we show fast convergence results in expectation and in almost sure sense. \tcm{This result imposes an integrability condition on the diffusion term to ensure a convergence rate of $\frac{1}{t^2}$ for the function values. This is highly desirable, as it represents the fastest convergence rate we can expect to achieve for gradient-based second-order dynamics featuring inertia when applied to a general convex function. Besides, these results match those in \cite{fast} when there is no noise.}

\begin{corollary}[Case $\frac{\alpha}{t}$] \label{corat}
Let $\nu\geq 2,\alpha>3$ and consider the dynamic \eqref{ISIHD-S} with initial data ${X_0},{V_0}\in \Lp^{\nu}(\Omega; \H)$, in the case where $\gamma(t)=\frac{\alpha}{t}$ and $\beta(t)=\frac{t}{\alpha-1}$. {Besides, consider that $f$ and $\sigma$ satisfy \eqref{H01} and \eqref{H}}. Moreover, let $f\in C^2(\H)$ and $t\mapsto t^2\sigma_{\infty}(t)\in \Lp^2([t_0,+\infty[)$. Then the solution trajectory $({X,V})\in S_{\H\times\H}^{\nu}[t_0]$ is unique and satisfies: \begin{enumerate}[label=(\roman*)]
    \item $f({X(t)})-\min f=o(t^{-2})$ a.s..
    \item $\mathbb{E}[f({X(t)})-\min f]=\mathcal{O}(t^{-2})$.
    \item $$\int_{t_0}^{\infty}t^3\Big\Vert\nabla f\left({X(t)}+\frac{t}{\alpha-1}{V(t)}\right)\Big\Vert^2dt<+\infty\quad a.s..$$
\end{enumerate}
\end{corollary}


\proof{}
Consider \eqref{inv} with $\Gamma(t)=\frac{t}{\alpha-1}$ and $\theta(t)=\frac{t^2-t_0^2}{2(\alpha-1)}$. Let $\tilde{\sigma}(s,\cdot)=\sigma(\theta^{-1}(s),\cdot)\sqrt{\Gamma(\theta^{-1}(s))}$. Notice that $t\mapsto t^2\sigma_{\infty}(t)\in \Lp^2([t_0,+\infty[)$ is equivalent to $s\mapsto s\tilde{\sigma}_{\infty}^2(s)\in \Lp^1(\R_+)$. We apply Theorem~\ref{3.4} to deduce that $$f({Z(s)})-\min f=o(s^{-1}) \text{ a.s..}$$
Using the time scaling $\tau\equiv \theta$ and then the averaging technique as in the proof of Theorem~\ref{trajectory}, we have that $$f({Y(t)})-\min f=o(t^{-2}) \text{ a.s..}$$
Moreover, it holds that $${X(t)}=\int_{t_0}^t {Y(u)}d\mu_t(u)+\xi(t).$$
\begin{enumerate}[label=(\roman*)]
    \item 
Now we prove the first point in the following way:
\begin{align*}
    t^2(f({X(t)})-\min f)&=t^2\left(f({X(t)})-f\left(\int_{t_0}^t {Y(u)}d\mu_t(u)\right)\right)+t^2\left(f\left(\int_{t_0}^t {Y(u)}d\mu_t(u)\right)-\min f\right).
\end{align*}
Let us bound the first term using the convexity of $f$: \begin{align*}
    t^2\left(f({X(t)})-f\left(\int_{t_0}^t {Y(u)}d\mu_t(u)\right)\right)&\leq t^2\Vert\nabla f({X(t)})\Vert\Vert\xi(t)\Vert\\
    &\leq t^2\Vert\xi(t)\Vert(L\Vert {X(t)}\Vert+\Vert\nabla f(0)\Vert)\\
    &\leq t^2\Vert\xi(t)\Vert\left(L\sup_{t\geq t_0}\Vert {X(t)}\Vert+\Vert\nabla f(0)\Vert\right).
\end{align*}
\smallskip

Let us recall that $\sup_{s\geq 0}\Vert {Z}(s)\Vert<+\infty$ a.s.. Due to the time scaling and averaging, it is direct to check that $\sup_{t\geq t_0}\Vert {X(t)}\Vert<+\infty$ a.s.. On the other hand, $\Vert\xi(t)\Vert=\mathcal{O}(t^{1-\alpha})$ a.s.. Therefore, we have \begin{equation}
    t^2\left(f({X(t)})-f\left(\int_{t_0}^t {Y(u)}d\mu_t(u)\right)\right)=\mathcal{O}(t^{3-\alpha})\quad a.s..
\end{equation}
Now let us bound the second term using Jensen's inequality, \begin{align*}
    t^2\left(f\left(\int_{t_0}^t {Y(u)}d\mu_t(u)\right)-\min f\right)&\leq t^2\left(\int_{t_0}^t [f({Y(u)})-\min f] d\mu_t(u)\right)\\
    &=\frac{t_0^{\alpha-1}}{t^{\alpha-3}}[f({Y(t_0)})-\min f]\\
    &+\frac{\alpha-1}{t^{\alpha-3}}\int_{t_0}^t u^{\alpha-4}(u^2(f({Y(u)})-\min f))du.
\end{align*}
In order to calculate the limit of this second term, let $a(t)=\frac{\alpha-1}{t}, b(u)=u^2(f({Y(u)})-\min f)$, by Lemma~\ref{ab} we have that $$\lim_{t\rightarrow\infty}\frac{\alpha-1}{t^{\alpha-1}}\int_{t_0}^t u^{\alpha-2}b(u)du=0 \quad a.s..$$
Since $\alpha>3$, we also have that \begin{equation}\label{alpha-3}
    \lim_{t\rightarrow\infty}\frac{\alpha-3}{t^{\alpha-3}}\int_{t_0}^t u^{\alpha-4}b(u)du=0\quad a.s..
\end{equation}
Therefore, we conclude that $$\lim_{t\rightarrow\infty}t^2(f({X(t)})-\min f)=0\quad a.s..$$
\item By Theorem~\ref{maxai} in the case $\gamma(t)=\frac{\alpha}{t}$, we have that $e^{-A(t)}=t_0^{\alpha-1}t^{1-\alpha}$ and $\theta(t)=\frac{t^2-t_0^2}{2(\alpha-1)}$. On the other hand $$I\left[\frac{1}{\theta}\right](t)=2(\alpha-1)^2t^{1-\alpha}\int_{t_0}^t\frac{u^{\alpha-2}}{u^2-t_0^2}=\mathcal{O}(t^{1-\alpha}+t^{-2}).$$
 Since $\alpha>3$, we have that $\mathcal{O}(t^{1-\alpha})$ is also $\mathcal{O}(t^{-2})$, and we conclude that $$\mathbb{E}(f({X(t)})-\min f)=\mathcal{O}(t^{-2}).$$
 \item This point follows directly from Theorem~\ref{thetagamma} in the case $\gamma(t)=\frac{\alpha}{t}$.
\end{enumerate}
\endproof

\subsection{Convergence rate under Polyak-{\L}ojasiewicz inequality}
In this subsection, we show a local convergence rate under Polyak-{\L}ojasiewicz inequality. The Polyak-{\L}ojasiewicz property is a special case of the {\L}ojasiewicz property (see \cite{loj1,loj2,loj3}) and is commonly used to prove linear convergence of gradient descent algorithms.

\begin{definition}[Polyak-{\L}ojasiewicz inequality]\label{def:lojq}
Let $f:\H\rightarrow\R$ be a differentiable function with $\calS\neq\emptyset$. Then, $f$ satisfies the Polyak-{\L}ojasiewicz (P{\L}) inequality on $\calS$, if there exists $r>\min f$ and $\mu > 0$ such that
\begin{equation}\label{li}
2\mu(f(x)-\min f)\leq \norm{\nabla  f(x)}^2,\quad \forall x \in  [\min f < f < r] ,
\end{equation}
 and we will write $f \in \PL_{\mu}(\calS)$. 

\end{definition}
\begin{theorem}\label{crpolyak}
Let $\nu\geq 2$ and consider the dynamic \eqref{ISIHD-S} with initial data ${X_0},{V_0}\in\Lp^{\nu}(\Omega;\H)$, where $f$ satisfies \eqref{H01}, and $\sigma$ satisfies \eqref{H}. Besides, $f\in \PL_{\mu}(\calS)$ and suppose that either $\H$ is finite dimensional or $f\in C^2(\H)$. Let also, $\gamma\equiv \sqrt{2\mu}$, $\beta\equiv\Gamma\equiv\frac{1}{\sqrt{2\mu}}$, and such that $\sigma_{\infty}\in \Lp^2([t_0,+\infty[)$.\smallskip

Then the solution trajectory $({X,V})\in S_{\H\times \H}^{\nu}[t_0]$ is unique. Moreover, letting $\delta>0$, then there exists $\td>t_0, \sd>0, K_{\mu,\delta},C_l,C_f>0$ such that:
 \begin{equation}
    \EE(f({X(t)})-\min f)\leq K_{\mu,\delta}e^{-\frac{\mu}{2}(t-\td)}+\frac{1}{\mu}l_{\delta}\pa{\frac{t+3\td-4t_0}{4\mu}}+C_f\sqrt{\delta}, \quad \forall t>\td,
\end{equation}
where $$l_{\delta}(s)=\frac{L}{2}\sigma_{\infty}^2(s)+C_l\sqrt{\delta}\frac{\sigma_{\infty}^2(s)}{2\sqrt{\int_{\sd}^s \sigma_{\infty}^2(u)du}}.$$

Besides, if $f\in \PL_{\mu}(\calS)$ holds on the entire space (i.e. $r=+\infty$), then we have that there exists $K_{\mu}>0$ such that:
 \begin{equation}
    \EE(f({X(t)})-\min f)\leq K_{\mu}e^{-\frac{\mu}{2}(t-t_0)}+\frac{L}{2\mu}\sigma_{\infty}^2\pa{\frac{t-t_0}{4\mu}}, \quad \forall t>t_0,
\end{equation}
\end{theorem}
\begin{remark}
    If $f$ is $\mu-$strongly convex, then $f\in \PL_{\mu}(\calS)$ holds on the entire space (i.e. $r=+\infty$). 
\end{remark}
\begin{remark}
    \tcm{For a proper discussion on the arbitrarily small (but non-zero) term $C_f\sqrt{\delta}$, we refer to \cite[Section 4]{mio}.}
\end{remark}
\proof{}
 Consider the dynamic \eqref{ISIHD-S} with $\gamma\equiv c, \beta\equiv\Gamma\equiv\frac{1}{c}$, where $c>0$ is a constant that will be fixed later.\smallskip
 
 Let us also define $\theta(t)\eqdef \int_{t_0}^t \Gamma(u)du=\frac{t-t_0}{c}$ and $\tilde{\sigma}(s,\cdot)\eqdef\sigma(\theta^{-1}(s),\cdot)\sqrt{\Gamma(\theta^{-1}(s))}$. Then $\sigma_{\infty}\in \Lp^2([t_0,+\infty[)$ is equivalent to $\tilde{\sigma}_{\infty}\in\Lp^2(\R_+)$.
Now consider the dynamic: \smallskip

\begin{equation}
\begin{cases}
d{Z(s)}&=-\nabla f({Z(s)})+\tilde{\sigma}(s,{Z(s)})dW(s), \quad s>0,\\
{Z(0)}&={X_0}+\Gamma(t_0){V_0}.
\end{cases}    
\end{equation}

Let $\delta>0$ and apply the result of \cite[Theorem~4.5(i-b)]{mio} (with coefficient $\sqrt{2\mu}$), that is, there exists $\sd>0$ such that for every $\lambda \in ]0,1[$,
\begin{equation}
\begin{aligned}
\EE\pa{f({Z(s)})-\min f}&\leq e^{-2\mu(s-\sd)}\EE(f({Z}(\sd))-\min f)\\
&+e^{-2\mu (1-\lambda)(s-\sd)}\pa{\frac{LC_{\infty}^2}{2}+C_lC_{\infty}\sqrt{\delta}}\\
&+\frac{l_{\delta}(\sd+\lambda(s-\sd))}{2\mu}+C_f\sqrt{\delta}, \qquad\qquad \forall s>\sd,
\end{aligned}
\end{equation}
where $C_{\infty},C_l,C_f>0$ and the establishment of $l_{\delta}$ are detailed in \cite[Section 4.2]{mio}, in particular $C_{\infty}=\sqrt{\int_{t_0}^{+\infty}\sigma_{\infty}^2(s)ds}$.\smallskip

Consider the time scaling $\tau\equiv\theta$, ${Y(t)}={Z}(\theta(t))$ and $\td>t_0$ such that $\theta(\td)=\sd$ (i.e. $\td=c\sd+t_0$), we have that:  \begin{equation}
\begin{aligned}\label{localy}
\EE\pa{f({Y(t)})-\min f}&\leq e^{-2\mu(\theta(t)-\sd)}\EE(f({Y}(\td))-\min f)\\
&+e^{-2\mu (1-\lambda)(\theta(t)-\sd)}\pa{\frac{LC_{\infty}^2}{2}+C_lC_{\infty}\sqrt{\delta}}\\
&+\frac{l_{\delta}(\sd+\lambda(\theta(t)-\sd))}{2\mu}+C_f\sqrt{\delta}, \qquad\qquad \forall t>\td.
\end{aligned}
\end{equation}
Let $a(t)=c$ and $A(t)=c(t-\td)$. Now, we consider the averaging as in \eqref{initia} but change the initial condition to $\td$. Thus, we have \begin{equation}
    {X(t)}=\int_{\td}^t {Y(u)}d\tilde{\mu}_t(u)+\tilde{\xi}(t),
\end{equation}
where $\tilde{\mu}_t$ is the probability measure on $[\td,t]$ defined by \begin{equation}\label{medida1}
    \tilde{\mu}_t=e^{-c(t-\td)}\delta_{\td}+ce^{c(u-t)}du,
\end{equation}
where $\delta_{\td}$ is the Dirac measure at $\td$ and \begin{equation}\label{chi1}
 \tilde{\xi}(t)\eqdef-\frac{1}{c}e^{-c(t-\td)}{V}(\td).   
\end{equation}
Then
$$\EE(f({X(t)})-\min f)=\EE\pa{f({X(t)})-f\pa{\int_{\td}^t {Y(u)}d\mu_t(u)}}+\EE\pa{f\pa{\int_{\td}^t {Y(u)}d\mu_t(u)}-\min f}.$$
We can bound the first term using convexity and Cauchy-Schwarz inequality in the following way \begin{align*}\EE\pa{f({X(t)})-f\pa{\int_{\td}^t {Y(u)}d\tilde{\mu}_t(u)}}&\leq\sqrt{\EE(\Vert\nabla f({X(t)})\Vert^2)} \sqrt{\EE(\Vert\tilde{\xi}(t)\Vert^2)}\\
&\leq \frac{\sqrt{\EE(\Vert {V}(\td)\Vert^2)}}{c}\sqrt{2\Vert\nabla f(0)\Vert^2+2L^2\EE\pa{\sup_{t\geq t_0}\Vert {X(t)}\Vert^2}}e^{-c(t-\td)},
\end{align*}
where 
$\EE(\sup_{t\geq t_0}\Vert {X(t)}\Vert^2)<+\infty$ as mentioned in Corollary~\ref{corat}.

On the other hand, we can bound the second term using Jensen's inequality and then \eqref{localy}
\begin{align*}
    \EE\pa{f\pa{\int_{\td}^t {Y(u)}d\tilde{\mu}_t(u)}-\min f}&\leq \int_{\td}^t \EE(f({Y(u)}-\min f))d\tilde{\mu}_t(u)\\
    &\leq \int_{\td}^t  e^{-2\mu(\theta(u)-\sd)}\EE(f({Y}(\td))-\min f)d\tilde{\mu}_t(u)\\
&+\int_{\td}^t e^{-2\mu (1-\lambda)(\theta(u)-\sd)}\pa{\frac{LC_{\infty}^2}{2}+C_lC_{\infty}\sqrt{\delta}}d\tilde{\mu}_t(u)\\
&+\int_{\td}^t\frac{l_{\delta}(\sd+\lambda(\theta(u)-\sd))}{2\mu}d\tilde{\mu}_t(u)+C_f\sqrt{\delta}  \\
&=\pa{\EE(f({Y}(\td))-\min f)+\pa{\frac{LC_{\infty}^2}{2}+C_lC_{\infty}\sqrt{\delta}}+\frac{l_{\delta}(\sd)}{2\mu}}e^{-c(t-\td)}\\
&+c\EE(f({Y}(\td))-\min f)e^{-ct}\int_{\td}^t e^{-2\mu(\theta(u)-\sd)} e^{cu}du\\
&+c\pa{\frac{LC_{\infty}^2}{2}+C_lC_{\infty}\sqrt{\delta}}e^{-ct}\int_{\td}^t e^{-2\mu (1-\lambda)(\theta(u)-\sd)}  e^{cu}du\\
&+\frac{c}{2\mu}e^{-ct}\int_{\td}^t l_{\delta}(\sd+\lambda(\theta(u)-\sd))e^{cu} du+C_f\sqrt{\delta},\quad \forall t>\td.
\end{align*}
We bound the first integral as follows: \begin{align*}
    e^{-ct}\int_{\td}^t e^{-2\mu(\theta(u)-\sd)} e^{cu}du\leq e^{2\mu\pa{\frac{t_0}{c}+\sd}}e^{-\frac{2\mu}{c}t}.
\end{align*}
The second integral in the same way \begin{align*}
    e^{-ct}\int_{\td}^t e^{-2\mu (1-\lambda)(\theta(u)-\sd)}  e^{cu}du\leq e^{2\mu(1-\lambda)\pa{\frac{t_0}{c}+\sd}}e^{-\frac{2\mu(1-\lambda)}{c}t}.
\end{align*}
To treat the third integral we are going to split the integral in two in order to find a useful convergence rate. Let us recall that $l_{\delta}\in\Lp^1([\sd,+\infty[)$ and that $l_{\delta}$ is decreasing. Let us define $\varphi_{\lambda,c,\delta}(t)\eqdef \sd+\lambda\pa{\frac{t+\td-2t_0}{2c}-\sd}$, then 
\begin{align*}
    e^{-ct}\int_{\td}^t l_{\delta}(\sd+\lambda(\theta(u)-\sd))e^{cu} du&=e^{-ct}\int_{\td}^{\frac{\td+t}{2}} l_{\delta}(\sd+\lambda(\theta(u)-\sd))e^{cu} du\\
    &+e^{-ct}\int_{\frac{\td+t}{2}}^t l_{\delta}(\sd+\lambda(\theta(u)-\sd))e^{cu} du\\
    &\leq \frac{c}{\lambda}e^{\frac{c\td}{2}}C_{\infty}e^{-\frac{ct}{2}}+l_{\delta}(\varphi_{\lambda,c,\delta}(t)).
\end{align*}
Now that we have bounded all the terms, we have the following bound \begin{align*}
    \EE(f({X(t)})-\min f)&\leq \frac{\sqrt{\EE(\Vert {V}(\td)\Vert^2)}}{c}\sqrt{2\Vert\nabla f(0)\Vert^2+2L^2\EE\pa{\sup_{t\geq t_0}\Vert {X(t)}\Vert^2}}e^{-c(t-\td)}\\
    &+\pa{\EE(f({Y}(\td))-\min f)+\pa{\frac{LC_{\infty}^2}{2}+C_lC_{\infty}\sqrt{\delta}}+\frac{l_{\delta}(\sd)}{2\mu}}e^{-c(t-\td)}\\
    &+c\EE(f({Y}(\td))-\min f)e^{2\mu\pa{\frac{t_0+\td}{c}+\sd}}e^{-\frac{2\mu}{c}(t-\td)}\\
    &+c\pa{\frac{LC_{\infty}^2}{2}+C_lC_{\infty}\sqrt{\delta}}e^{2\mu(1-\lambda)\pa{\frac{t_0+\td}{c}+\sd}}e^{-\frac{2\mu(1-\lambda)}{c}(t-\td)}\\
    &+\frac{c}{2\mu}\pa{\frac{c}{\lambda}C_{\infty}e^{-\frac{c(t-\td)}{2}}+l_{\delta}(\varphi_{\lambda,c,\delta}(t))}+C_f\sqrt{\delta}, \quad \forall t>\td.
\end{align*}

Letting $\lambda=\frac{1}{2}$ and $c=\sqrt{2\mu}$ we obtain
\begin{align*}
    \EE(f({X(t)})-\min f)&\leq \frac{\sqrt{\EE(\Vert {V}(\td)\Vert^2)}}{\sqrt{2\mu}}\sqrt{2\Vert\nabla f(0)\Vert^2+2L^2\EE\pa{\sup_{t\geq t_0}\Vert {X(t)}\Vert^2}}e^{-\sqrt{2\mu}(t-\td)}\\
    &+\pa{\EE(f({Y}(\td))-\min f)+\pa{\frac{LC_{\infty}^2}{2}+C_lC_{\infty}\sqrt{\delta}}+\frac{l_{\delta}(\sd)}{2\mu}}e^{-\sqrt{2\mu}(t-\td)}\\
    &+\sqrt{2\mu}\EE(f({Y}(\td))-\min f)e^{2\mu\pa{\frac{t_0+\td}{\sqrt{2\mu}}+\sd}}e^{-\sqrt{2\mu}(t-\td)}\\
    &+\sqrt{2\mu}\pa{\frac{LC_{\infty}^2}{2}+C_lC_{\infty}\sqrt{\delta}}e^{\mu\pa{\frac{t_0+\td}{\sqrt{2\mu}}+\sd}}e^{-\frac{\sqrt{2\mu}}{2}(t-\td)}\\
    &+2 C_{\infty}e^{-\frac{\sqrt{2\mu}(t-\td)}{2}}+\frac{1}{\sqrt{2\mu}}l_{\delta}(\varphi_{\frac{1}{2},\sqrt{2\mu},\delta}(t))+C_f\sqrt{\delta}, \quad \forall t>\td.
\end{align*}
Letting $K_{\mu,\delta}\eqdef\sqrt{2\mu}\pa{\frac{LC_{\infty}^2}{2}+C_lC_{\infty}\sqrt{\delta}}e^{\mu\pa{\frac{t_0+\td}{\sqrt{2\mu}}+\sd}}+2 C_{\infty}$, we conclude that \begin{equation}
    \EE(f({X(t)})-\min f)\leq K_{\mu,\delta}e^{-\frac{\sqrt{2\mu}}{2}(t-\td)}+\frac{1}{\sqrt{2\mu}}l_{\delta}(\varphi_{\frac{1}{2},\sqrt{2\mu},\delta}(t))+C_f\sqrt{\delta}, \quad \forall t>\td.
\end{equation}
\endproof

\section{From weak to strong convergence under general \texorpdfstring{$\gamma$}{} and \texorpdfstring{$\beta\equiv\Gamma$}{}}
\subsection{General result}
We consider the Tikhonov regularization of the dynamic \eqref{ISIHDN-S}, i.e., for $t>0$, 
\begin{align} \label{ISIHDN-S-TA}\tag{$\mathrm{S-ISIHD_{NS}-TA}$}
\begin{cases}
d{X(t)}&={V(t)}dt,  \nonumber\\
d{V(t)}&\in -\gamma(t){V(t)}dt-\partial F({X(t)}+\beta(t){V(t)})dt-\varepsilon(t)({X(t)}+\beta(t){V(t)})dt\\
&\phantom{=}+\sigma(t,{X(t)}+\beta(t){V(t)})dW(t),\\
{X(t_0)}&={X_0}, \quad {V(t_0)}={V_0}. \nonumber
\end{cases}
\end{align}
We show some conditions (on $\gamma,\beta,\varepsilon$) under which we can obtain strong convergence of the trajectory.

\begin{theorem}\label{trajectory1}
 Consider that $\gamma:[t_0,+\infty[\rightarrow\R_+$ satisfies \eqref{H1}. Besides, $F=f+g$ and $\sigma$ satisfy assumptions \eqref{H0} and \eqref{H}. Moreover, suppose that $g$ satisfies \eqref{Hl} and let $\nu\geq 2$. Consider \eqref{ISIHDN-S-TA} with $\beta\equiv\Gamma$ and initial data ${X_0},{V_0}\in\Lp^{\nu}(\Omega;\H)$. 

 
  Then, there exists a unique solution $({X,V})\in S_{\H\times \H}^{\nu}[t_0]$ of \eqref{ISIHDN-S-TA}.  Additionally, let ${x^{\star}}\eqdef \proj_{\calS_F}(0)$ be the minimum norm solution, and for $\varepsilon>0$ let $x_{\varepsilon}$ be the unique minimizer of $F_{\varepsilon}(x)\eqdef F(x)+\frac{\varepsilon}{2}\|x\|^2$. If we suppose that $\Gamma\sigma_{\infty}\in \Lp^2([t_0,+\infty[)$, and that $\varepsilon:[t_0,+\infty[\rightarrow\R_+$ satisfies the conditions:

\vspace{2mm}
\begin{enumerate}[label=\rm{$(T_{\arabic*}^{'})$}]
    \item \label{t1'} $\varepsilon (t) \to 0$ \textit{as} $t \to +\infty$;
    \item \label{t2'} $\displaystyle{\int_{t_0}^{+\infty} \varepsilon (t)\Gamma(t) dt = +\infty}$;
    \item \label{t3'} $\displaystyle{ \int_{t_0}^{+\infty} \varepsilon (t) \Gamma(t)\left(    \|{x^{\star}} \|^2 -  \|x_{\varepsilon (t)} \|^2  \right)dt  <+\infty}. $
\end{enumerate}

\noindent Then $\slim_{t\rightarrow +\infty} {X(t)}={x^{\star}}$ a.s., and $ {V(t)}=o\left(\frac{1}{\Gamma(t)}\right)$ a.s..
\end{theorem}
\proof{}
Let $s_0>0$, $\theta(t)\eqdef s_0+\int_{t_0}^t \Gamma(u)du$; $\tilde{\varepsilon}(t)=\varepsilon( \theta^{-1}(t))
$; and $\tilde{\sigma}(s,\cdot)\eqdef\sigma(\theta^{-1}(s),\cdot)\sqrt{\Gamma(\theta^{-1}(s))}$. Then $\varepsilon$ satisfying~\ref{t1'},\ref{t2'}, and~\ref{t3'} is equivalent to $\tilde{\varepsilon}$ satisfying~\ref{t1},\ref{t2}, and~\ref{t3} defined in Theorem \ref{converge20}. Besides, $\Gamma\sigma_{\infty}\in \Lp^2([t_0,+\infty[)$ is equivalent to $\tilde{\sigma}_{\infty}\in\Lp^2(\R_+)$.
Consider the dynamic: \smallskip

\begin{equation}
\begin{cases}\label{inv1}
d{Z(s)}&\in -\partial F({Z(s)})-\tilde{\varepsilon}(s){Z(s)}+\tilde{\sigma}(s,{Z(s)})dW(s),\quad s>s_0,\\
{Z(s_0)}&={X_0}+\Gamma(t_0){V_0}.
\end{cases}    
\end{equation}
By Theorem~\ref{converge20}, we have that there exists a unique solution ${Z}\in S_{\H}^{\nu}[s_0]$, and that $\lim_{s\rightarrow\infty} {Z(s)} = {x^{\star}}$ a.s. (Recall that ${x^{\star}}\eqdef \proj_{S_F}(0)$). Using the time scaling $\tau\equiv \theta$ and the averaging described at the beginning of this section, we end up with the dynamic \eqref{ISIHDN-S-TA} in the case where $\beta\equiv\Gamma$. The existence and uniqueness of solution, and the fact that $({X,V})\in S_{\H\times\H}^{\nu}[t_0]$ goes analogously as in  the proof of Theorem~\ref{trajectory}.\smallskip

 Now we prove the claim, since $\lim_{s\rightarrow\infty} {Z(s)} = {x^{\star}}$ a.s., this implies directly that $\lim_{t\rightarrow\infty} {Y(t)} = {x^{\star}}$ a.s.. Besides, we have the relation \eqref{aver}, i.e. $$ {X(t)}=\int_{t_0}^t {Y(u)}d\mu_t(u)+\xi(t),$$
where $\mu_t$ and $\xi$ are defined in \eqref{medida} and \eqref{chi}, respectively. 
Consequently, we have \begin{align*}
    \Vert  {X(t)}-{x^{\star}}\Vert&\leq \bignorm{ {X(t)}-\int_{t_0}^t {Y(u)}d\mu_t(u)}+\bignorm{\int_{t_0}^t {Y(u)}d\mu_t(u)-{x^{\star}}}\\
    &\leq \Vert\xi(t)\Vert+\bignorm{\int_{t_0}^t {Y(u)}d\mu_t(u)-{x^{\star}}}.
\end{align*}

\noindent Let $a(t)=\frac{1}{\Gamma(t)}$ and $A(t)=\int_{t_0}^t\frac{du}{\Gamma(u)}$. By Lemma~\ref{1gam}, we have that $\lim_{t\rightarrow\infty}\Vert\xi(t)\Vert=0$. On the other hand \begin{align*}
    \bignorm{\int_{t_0}^t {Y(u)}d\mu_t(u)-{x^{\star}}}&=\bignorm{\int_{t_0}^t ({Y(u)}-{x^{\star}})d\mu_t(u)}\\
    &\leq\int_{t_0}^t \Vert {Y(u)}-{x^{\star}}\Vert d\mu_t(u)\\
    &=e^{-A(t)}\Vert {Y(t_0)}-{x^{\star}}\Vert+e^{-A(t)}\int_{t_0}^t a(u)e^{A(u)}\Vert {Y(u)}-{x^{\star}}\Vert du.
\end{align*}
Let $b(u)=\Vert {Y(u)}-{x^{\star}}\Vert$. Since we already proved that $\lim_{u\rightarrow\infty}b(u)=0$ a.s., and we have that $a\notin \Lp^1([t_0,+\infty[)$ by Lemma~\ref{1gam}, we utilize Lemma~\ref{ab} with our respective $a,b$ functions. This let us conclude that $$ \lim_{t\rightarrow\infty}\bignorm{\int_{t_0}^t {Y(u)}d\mu_t(u)-{x^{\star}}}=0 \quad a.s..$$
Thus, $\lim_{t\rightarrow\infty} {X(t)} = {x^{\star}}$ a.s.. Finally, since $${Y(t)}={X(t)}+\Gamma(t){V(t)},$$
and the fact that ${X}$ and ${Y}$ have (a.s.) the same limit, we conclude that $$\lim_{t\rightarrow\infty}\Gamma(t){V(t)}=0 \quad a.s..$$

\endproof
\subsection{Concrete cases}
In order to give some conditions when~\ref{t1'},~\ref{t2'}, and~\ref{t3'} of Theorem~\ref{trajectory1} are satisfied we need to introduce the following definition:
\begin{definition}[H\"olderian error bound]
Let $f:\H\rightarrow\R$ be a proper function such that $\calS \neq \emptyset$. $f$ satisfies a H\"olderian (or power-type) error bound inequality on $\calS$ with exponent $p \geq 1$, if there exists $\kappa>0$ and $r>\min f$ such that
\begin{equation}\label{eq:errbnd}
f(x)-\min f \geq \kappa\dist(x,\calS)^p, \quad \forall x \in [\min f \leq f \leq r] ,
\end{equation}
and we will write $f \in \EB^p(\calS)$,
\end{definition}
\begin{remark}
    Let $f:\H\rightarrow\R$ be a differentiable function such that $\calS\neq\emptyset$. If $f$ satisfies the Polyak-{\L}ojasiewicz inequality on $\calS$, then $f$ satisfies a H\"olderian error bound inequality with exponent $p=2$.
\end{remark}

\begin{theorem}\label{practical}
Consider the setting of Theorem~\ref{trajectory1} and suppose that  $F=f+g \in \EB^p(\calS_F)$. Let $s_0>0$ and denote $\theta(t)\eqdef s_0+\int_{t_0}^t\Gamma(s)ds$, then taking the Tikhonov parameter
 $\varepsilon(t) = \frac{1}{\theta^r(t)}$ with
$$
1 \geq r > \frac{2p}{2p+1},
$$
then the three conditions~\ref{t1'},~\ref{t2'}, and~\ref{t3'} of Theorem~\ref{trajectory1}   are satisfied simultaneously.
In particular, for any solution $({X,V})\in S_{\H\times \H}^{\nu}[t_0]$ of \eqref{trajectory1}, we get almost sure (strong) convergence of ${X(t)}$ to the minimal norm solution named ${x^{\star}}=P_{\calS_F}(0)$ and that ${V(t)}=o\left(\frac{1}{\Gamma(t)}\right)$.
\end{theorem}
\proof{}
    We proceed as in the proof of Theorem~\ref{trajectory1} and arrive to the dynamic \eqref{inv1}, since $\tilde{\varepsilon}(t)=\varepsilon(\theta^{-1}(t))=\frac{1}{t^r}$, the proof goes as in \cite[Theorem 4.8]{mio2}.
\endproof

\begin{theorem}\label{importante1}
    Let $\nu\geq 2$, $f\in \Gamma_0(\H)\cap C_L^2(\H)$ such that $\calS$ is nonempty, and also $f\in \EB^p(\calS)$, $\sigma$ satisfying \eqref{H}, and $\Gamma\sigma_{\infty}\in\Lp^2([t_0,+\infty[)$ and is non-increasing. Let us consider $\varepsilon(t)=\frac{1}{t^r}$ where $0<r<1$, then we evaluate \eqref{ISIHDN-S-TA} in the case where $\gamma$ satisfies \eqref{H1}, $g\equiv 0$, $\beta\equiv\Gamma$,  and with initial data ${X_0},{V_0}\in\Lp^{\nu}(\Omega;\H)$, i.e., for $t>t_0$,

\begin{equation} \label{g0} 
\begin{cases}
d{X(t)}&={V(t)}dt,  \\
d{V(t)}&=-\gamma(t){V(t)}dt-\nabla f({X(t)}+\Gamma(t){V(t)})dt-\varepsilon(t)({X(t)}+\Gamma(t){V(t)})\\
&\phantom{=}+\sigma(t,{X(t)}+\Gamma(t){V(t)})dW(t),\\
{X(t_0)}&={X_0}, \quad {V(t_0)}={V_0}. 
\end{cases}
\end{equation}
For $\varepsilon>0$, let us define $f_{\varepsilon}(x)\eqdef f(x)+\frac{\varepsilon}{2}\Vert x\Vert^2$, and let $x_{\varepsilon}$ be the unique minimizer of $f_{\varepsilon}$. Moreover, let $s_0>0$ and for $s_1>s_0$ consider, \begin{equation}\label{eqdefr}
    R(s)\eqdef e^{-\frac{s^{1-r}}{1-r}}\int_{s_1}^s e^{\frac{u^{1-r}}{1-r}}\sigma_{\infty}^2(\theta^{-1}(u))\Gamma(\theta^{-1}(u)) du,
\end{equation}
where $\theta(t)\eqdef s_0+\int_{t_0}^t \Gamma(u)du$.
\smallskip
Let also ${x^{\star}}\eqdef \proj_{\calS}(0)$, $A(s)\eqdef \int_{s_1}^s\frac{du}{\Gamma(u)}$, and $t_1\eqdef\theta^{-1}(s_1)$. Then, the solution trajectory $({X,V})\in S_{\H\times\H}^{\nu}[t_0]$ is unique, and we have that:  	
\begin{enumerate}[label=(\roman*)]
    \item $R(\theta(t))\rightarrow 0 \mbox{ as } \; t \to +\infty.$ \label{error}\\
    \item Let $\bar{\sigma}(t)=\Gamma(t)\sigma_{\infty}^2(t)$, then $$R(\theta(t))=\mathcal{O}\left(\mathrm{exp}(-\theta^r(t)(1-2^{-r}))+\theta^r(t)\bar{\sigma}\left(\frac{s_1+\theta(t)}{2}\right)\right).$$ Moreover, if $\bar{\sigma}(t)=\mathcal{O}(\theta^{-\Delta}(t))$ for $\Delta>1$, then $R(\theta(t))=\mathcal{O}(\theta^{r-\Delta}(t))$.
    \end{enumerate}\smallskip
    
    Besides, we have the following convergence rate in expectation:
    \begin{enumerate}[resume*]
        
	\item \label{n5} For the values, we have:
	$$\EE[f({X(t)})-\min (f)]= \mathcal O \left( \max\{e^{-A(t)},I[h_1](t)\}   \right),$$ where 
 $h_1(t)=\frac{1}{\theta^r(t)}+R(\theta(t))$. 
	
    \item  \label{n6} And for the trajectory, we obtain: $$\EE[\|{X(t)} -{x^{\star}}\|^2]=\mathcal{O}\left(\max\{e^{-A(t)},I[h_2](t)\}\right),$$  where 
    $h_2(t)=\theta^{r-1}(t)+\theta^{-\frac{r}{p}}(t)+\theta^r(t)R(\theta(t))$. 
    
	\end{enumerate}
    
\end{theorem}
\proof{}
    We proceed as in the proof of Theorem~\ref{trajectory} and define analogously $\tilde{\sigma}$, $\tilde{\varepsilon}$, we also consider the dynamic \eqref{inv}. By \cite[Theorem 4.11]{mio2} we obtain that $$R(s)=e^{-\frac{s^{1-r}}{1-r}}\int_{s_1}^s e^{\frac{u^{1-r}}{1-r}}\tilde{\sigma}_{\infty}^2(u) du,$$
    where $\tilde{\sigma}_{\infty}^2\in\Lp^2([s_0,+\infty[)$, satisfies the following:
    \begin{itemize}
        \item $R(s)\rightarrow+\infty$ as $t\rightarrow+\infty$.
        \item $R(s)=\mathcal{O}\left(\mathrm{exp}(-s^r(1-2^{-r}))+s^r\tilde{\sigma}_{\infty}^2\left(\frac{s_1+s}{2}\right)\right)$. Moreover if $\tilde{\sigma}_{\infty}^2(s)=\mathcal{O}(s^{-\Delta})$ for $\Delta>1$, then $R(s)=\mathcal{O}(s^{r-\Delta})$. 
    \end{itemize}
    And evaluating at $s=\theta(t)$ we obtain the first two items of the theorem. For the third and fourth items we used that \begin{itemize}
        \item $\EE[f({Z(s)})-\min (f)]= \mathcal O \left( \displaystyle\frac{1}{s^{r} }+R(s)   \right)$.
        \item $\EE[\|{Z(s)} -{x^{\star}}\|^2]=\mathcal{O}\left(\dfrac{1}{ s^{1-r}}+\dfrac{1}{ s^{\frac{r}{p}}}+s^{r}R(s)\right) .$ 
    \end{itemize}
    Then, proceeding as in the proof of Theorem~\ref{maxai}, we obtain the desired results.
\endproof
\begin{corollary}
    Consider Theorem~\ref{practical} in the case where $\gamma(t)=\frac{\alpha}{t}$ for $\alpha>1$, $\beta(t)= \frac{t}{\alpha-1}$ then we have that:
    \begin{enumerate}
        \item If $\sigma_{\infty}^2(t)=\mathcal{O}(t^{-2(\Delta+1)})$ for $\Delta>1$, and $\alpha\neq \{1+2r,1+2(\Delta-r)\}$ , then
	$$\EE[f({X(t)})-\min (f)]= \mathcal O \left( \max\{t^{-(\alpha-1)},t^{-2r},t^{-2(\Delta-r)}\}   \right).$$
 In particular, if $\alpha>3$
        \item If $\sigma_{\infty}^2(t)=\mathcal{O}(t^{-2(\Delta+1)})$ for $\Delta>\max\{1,2r\}$, and $\alpha\neq \{3-2r,1+\frac{2r}{p},1+2(2r-\Delta)\}$, then $$\EE[\|{X(t)} -{x^{\star}}\|^2]=\mathcal{O}\left(\max\{t^{-(\alpha-1)},t^{-2(1-r)},t^{-\frac{2r}{p}},t^{-2(2r-\Delta)}\}\right).$$
     \end{enumerate}
\end{corollary}

\section{Conclusion}\label{sec:conclusion}

This work uncovers the global and local convergence properties of trajectories of the Inertial System with Implicit Hessian-driven Damping under stochastic errors both in the smooth and non-smooth setting. The aim is to solve convex optimization problems with noisy gradient input with vanishing variance. We have shed light on these properties and provided a comprehensive local and global complexity analysis both in the case where the Hessian damping parameter $\beta$ was dependent on the geometric damping $\gamma$ and when it was zero. We believe that this work, along with the technique of time scaling and averaging, paves the way for important extensions and research avenues. Among them, we mention extension to the situation where the drift term is a non-potential co-coercive operator.







\appendix
\tcm{
\section{SDEs are better approximations to stochastic inertial algorithms}\label{sec:sdeode}
In this section, we will argue that \eqref{ISIHD-S} is an accurate continuous-time model of a natural stochastic inertial algorithm with an accuracy error that scales as $\mathcal{O}(h)$, and this is much better than \eqref{ISIHD} whose accuracy is only $\mathcal{O}(\sqrt{h})$. This generalizes \cite[Proposition~2.1]{sdemodel} and holds for a more general model of the diffusion coefficient and under weaker regularity assumptions. The proof is reminiscent of strong consistency bounds of Euler-type methods for nonlinear SDEs \cite{Kloeden92}, that we will refine by exploiting the structure of our SDE \eqref{ISIHD-S}. We also restrict our discussion to the finite dimensional case where $\H=\K=\R^d$.

Given a step-size $h>0$, the Euler-Maruyama discretization applied to \eqref{ISIHD-S} computes approximations $X_k$ and $V_k$  of $X(t_k)$ and $V(t_k)$, where $t_k = t_0+k h$, $k \in \N$, by initializing $(X_0,V_0)=(X(t_0),V(t_0))$ and forming the following iterative scheme: 
\begin{equation}\label{stochnest}
    \begin{cases}
        X_{k+1}&=X_k+hV_k,\\
        V_{k+1}&=(1-\gamma_k h)V_k-h\nabla f(X_k+\beta_k V_k) + \sqrt{h} \sigma_k G_k,
    \end{cases}
\end{equation}
where $G_k \sim \mathcal{N}(0,I_d)$, $\gamma_k = \gamma(t_k)$, $\sigma_k = \sigma(t_k,X_k+\beta_k V_k)$ and $\beta_k = \beta(t_k)$. This is a stochastic version of the algorithm proposed in \cite{alecsa}. Setting $Y_k=(X_k,V_k)$, it will be convenient to rewrite \eqref{stochnest} in the product space $\R^{2d}$ as
\begin{equation}
    \begin{cases}
        Y_{k+1}&=Y_k+h\Phi(t_k,Y_k)+\sqrt{h}{\varsigma}(t_k,Y_k)({W}(t_{k+1})-{W}(t_k)),\\
        Y_0&=(X_0,V_0) ,
    \end{cases}
\end{equation}
where ${W}$ is a $\R^{2d}-$valued Brownian motion and 
\[
\Phi(t,(x,v))=(v,-\gamma(t)v-\nabla f(x+\beta(t)v)), \quad \text{ and }\quad 
{\varsigma}(t,(x,v))=\begin{pmatrix}
    0_{d\times d} & 0_{d\times d}\\
    0_{d\times d} & \sigma(t,x+\beta(t)v)
\end{pmatrix}.
\]
This then motivates the SDE
\begin{equation}\label{ISIHD-S-product}
    \begin{cases}
        dY(t)&=\Phi(t,Y(t))dt+\sqrt{h}{\varsigma}(t,Y(t))d{W}(t),\\
        Y(t_0)&=(X_0,V_0),
    \end{cases}
\end{equation}
where we set $Y(t)=(X(t),V(t))$ which is \eqref{ISIHD-S} with an extra $\sqrt{h}$ in front the diffusion. 

As classical in numerical analysis of evolution equations, we define the continuous-time piece-wise linear extension of the sequence $(Y_k)_{k \in \N}$
\begin{align*}
\overline{Y}(t) 
&\eqdef Y_k + (t-t_k)\Phi(t_k,Y_k) + \sqrt{h}{\varsigma}(t_k,Y_k)({W}(t)-{W}(t_k)) \qforq t \in [t_k,t_{k+1}[\\
&= Y_0 + \int_{t_0}^t \widehat{\Phi}(s,\widehat{Y}(s))ds + \sqrt{h}\int_{t_0}^t \widehat{\varsigma}(s,\widehat{Y}(s))dW(s) ,
\end{align*}
where for $t \in [t_k,t_{k+1}[$, $\widehat{Y}(t) \eqdef Y_k$, $\widehat{\Phi}(t,\widehat{Y}(t)) \eqdef \Phi(t_k,Y_k)$, $\widehat{\varsigma}(t,\widehat{Y}(t)) \eqdef {\varsigma}(t_k,Y_k)$. We, also define $\widehat{\gamma}(t)$ and $\widehat{\beta}(t)$ similarly. We thus have $\widehat{\sigma}(t,\widehat{X}(t)+\widehat{\beta}(t)\widehat{V}(t)) = \sigma_k$ for $t \in [t_k,t_{k+1}[$.

Our result requires the following assumption.
\begin{assumption}\label{assump:isgfconsist}
$f:\R^d\rightarrow\R$ is continuously differentiable with $L$-Lipschitz continuous gradient. For $T > t_0$, $\gamma$ and $\beta$ are respectively $L_\gamma$- and $L_\beta$-Lipschitz continuous on $[t_0,T]$. $\sigma$ verifies, $\forall t,s \in [t_0,T]$ and $\forall x,x' \in \R^d$,
\[
\norm{\sigma(t,x)-\sigma(t,x')} \leq L_\sigma\norm{x-x'}, \norm{\sigma(t,x)} \leq K_\sigma(1+\norm{x}) \qandq \norm{\sigma(t,x)-\sigma(s,x)} \leq L_\sigma(1+\norm{x})|t-s|^{1/2} .
\]
where all constants $L_\gamma,L_\beta,L_\sigma,K_\sigma$ do not depend on $h$.
\end{assumption} 
All these assumptions are verified in the instances studied in Section~\ref{sec:scaling} when $f$ is smooth. Observe that these assumptions also ensure existence and uniqueness of the solution $(X(t),V(t))$ to \eqref{ISIHD-S-product} (hence \eqref{ISIHD-S}).

We have the strong consistency results which shows that \eqref{ISIHD-S-product} is a better continuous-time approximation to \eqref{stochnest} than \eqref{ISIHD}.
\begin{proposition}\label{prop:isgfconsist}
Suppose that Assumption~\ref{assump:isgfconsist} holds and that $X_0,V_0\in\Lp^2(\Omega;\H)$. Let $(X(t),V(t))$ be the solution to \eqref{ISIHD-S-product} and $(x(t),v(t)=\dot{x}(t))$ the solution to \eqref{ISIHD}. Then the iterates of algorithm \eqref{stochnest} satisfy, as $h \to 0$,
\begin{align}
\EE\sbrac{\sup_{0 \leq k \leq N-1}\norm{X_k - X(t_k)} + \norm{V_k - V(t_k)}} = \mathcal{O}(h),  \label{eq:isgfconsist}\\
\EE\sbrac{\sup_{0 \leq k \leq N-1}\norm{X_k - x(t_k)} + \norm{V_k - v(t_k)}} = \mathcal{O}(\sqrt{h}) \label{eq:igfconsist} .
\end{align}
\end{proposition}

\begin{proof}
By Assumption~\ref{assump:isgfconsist}, it is not difficult to see that $\Phi$ and $\varsigma$ verify the following the Lipschitz continuity and linear growth properties,
\begin{equation}\label{eq:liplingrowPhisig}
\begin{aligned}
\norm{\Phi(t,y) - \Phi(t,z)}^2 &\leq \max(4L^2,1+2\gamma_{\max}^2+4L^2\beta_{\max}^2)\norm{y-z}^2, &\forall t \in [t_0,T], y,z \in \R^d , \\
\norm{\Phi(t,y) - \Phi(s,y)}^2 &\leq (2L_\gamma^2+4L^2L_\beta^2)\norm{y}^2|t-s|^2, &\forall s,t \in [t_0,T], y \in \R^d , \\
\norm{\Phi(t,y)}^2 &\leq \max(8L^2,1+2\gamma_{\max}^2+8L^2\beta_{\max}^2)\norm{y}^2 + 4 \norm{\nabla f(0)}^2, &\forall t \in [t_0,T], y \in \R^d . 
\end{aligned}
\end{equation}

In the rest of the proof, $C$ is any positive constant that does not depend on $h$ (and may depend on $d$, $T-t_0$, the different Lipschitz and growth constants in Assumption~\ref{assump:isgfconsist}), and which may change from one line to another. Using Jensen's inequality, Doob's martingale inequality and \cite[Theorem~1.5.21]{mao}, we have for any $\tau \in [t_0,T]$
\begin{align}
&\EE\sbrac{\sup_{t_0 \leq t \leq \tau}\norm{\overline{Y}(t) - Y(t)}^2} \\
=& 
\EE\sbrac{\sup_{t_0 \leq t \leq \tau}\norm{\int_{t_0}^t \pa{\widehat{\Phi}(s,\widehat{Y}(s)) - {\Phi}(s,{Y}(s))}ds + \sqrt{h}\int_{t_0}^t \pa{\widehat{\varsigma}(s,\widehat{Y}(s)) - {\varsigma}(s,{Y}(s))}dW(s)}^2} \label{eq:consistmainineq}\\
\leq& C \EE\sbrac{\int_{t_0}^\tau\norm{\widehat{\Phi}(s,\widehat{Y}(s)) - {\Phi}(s,{Y}(s))}^2 ds} + 
Ch\EE\sbrac{\int_{t_0}^\tau \norm{\widehat{\sigma}(s,\widehat{X}(s)+\widehat{\beta}(s)\widehat{V}(s)) - {\sigma}(s,{X}(s)+{\beta}(s){V}(s))}^2ds} \nonumber.
\end{align}
Using Jensen's inequality, \eqref{eq:liplingrowPhisig} and Assumption~\ref{assump:isgfconsist} on $\sigma$, we get
\begin{align*}
&\EE\sbrac{\sup_{t_0 \leq t \leq \tau}\norm{\overline{Y}(t) - Y(t)}^2} \\
\leq& C\EE\sbrac{\int_{t_0}^\tau\norm{\widehat{\Phi}(s,\widehat{Y}(s)) - {\Phi}(s,\widehat{Y}(s))}^2 ds} + 
Ch\EE\sbrac{\int_{t_0}^\tau \norm{\widehat{\sigma}(s,\widehat{X}(s)+\widehat{\beta}(s)\widehat{V}(s)) - {\sigma}(s,\widehat{X}(s)+\widehat{\beta}(s)\widehat{V}(s))}^2ds} \\
& C(1+h)\int_{t_0}^\tau\EE\sbrac{\sup_{t_0 \leq r \leq s}\norm{\overline{Y}(r) - Y(r)}^2} ds +
C(1+h)\EE\sbrac{\int_{t_0}^\tau\norm{\overline{Y}(t) - \widehat{Y}(t)}^2 ds} \\
&+ Ch\EE\sbrac{\sup_{t_0 \leq t \leq T}\norm{Y(t)}^2}\int_{t_0}^\tau |\beta(s)-\widehat{\beta}(s)|^2 ds ,
\end{align*}
where the expectation in the last display is bounded by \cite[Theorem~4.5.4]{Kloeden92}. Let $N = \lfloor (T-t_0)/h \rfloor$. Lipschitz continuity of $\beta$ gives
\[
\int_{t_0}^\tau |\beta(s)-\widehat{\beta}(s)|^2 ds \leq \sum_{k=0}^{N-1}\int_{t_k}^{t_{k+1}} |\beta(s)-\beta(t_k)|^2 ds \leq L_\beta^2 \sum_{k=0}^{N-1}\int_{t_k}^{t_{k+1}} (s-t_k)^2 ds = \frac{L_\beta^2(T-t_0)}{3} h^2 .
\]
Now, we have for any $s \in [t_k,t_{k+1}[$
\[
\overline{Y}(s) - \widehat{Y}(s) = -(s-t_k)\Phi(t_k,Y_k) - \sqrt{h}{\varsigma}(t_k,Y_k)({W}(s)-{W}(t_k)) .
\]
It then follows from \eqref{eq:liplingrowPhisig}, Assumption~\ref{assump:isgfconsist} on $\sigma$ and \cite[Theorem~4.5.4]{Kloeden92} that
\[
\EE\sbrac{\int_{t_0}^\tau\norm{\overline{Y}(t) - \widehat{Y}(t)}^2 ds} \leq \sum_{k=0}^{N-1}\int_{t_k}^{t_{k+1}} \EE\sbrac{\norm{(s-t_k)\Phi(t_k,Y_k) + \sqrt{h}{\varsigma}(t_k,Y_k)({W}(s)-{W}(t_k))}^2} ds \leq C h^2 .
\]
With similar arguments, we have
\[
\EE\sbrac{\int_{t_0}^\tau\norm{\widehat{\Phi}(s,\widehat{Y}(s)) - {\Phi}(s,\widehat{Y}(s))}^2 ds} =
\sum_{k=0}^{N-1}\int_{t_k}^{t_{k+1}} \EE\sbrac{\norm{{\Phi}(t_k,{Y}_k) - {\Phi}(s,{Y}_k)}^2} ds \leq C h^2 
\]
and using Assumption~\ref{assump:isgfconsist} on $\sigma$ 
\begin{multline*}
\EE\sbrac{\int_{t_0}^\tau \norm{\widehat{\sigma}(s,\widehat{X}(s)+\widehat{\beta}(s)\widehat{V}(s)) - {\sigma}(s,\widehat{X}(s)+\widehat{\beta}(s)\widehat{V}(s))}^2ds} \\
= \sum_{k=0}^{N-1}\int_{t_k}^{t_{k+1}} \EE\sbrac{\norm{{\sigma}(t_k,X_k+\beta_kV_k) - {\sigma}(s,X_k+\beta_kV_k)}^2} ds \leq C h . 
\end{multline*}
Collecting all bounds, we get
\begin{align*}
\EE\sbrac{\sup_{t_0 \leq t \leq \tau}\norm{\overline{Y}(t) - Y(t)}^2} \leq 
C(h^2+h^3) + C(1+h)\int_{t_0}^\tau\EE\sbrac{\sup_{t_0 \leq r \leq s}\norm{\overline{Y}(r) - Y(r)}^2} ds .
\end{align*}
Applying the Jensen and Gronwall inequalities then gives
\[
\EE\sbrac{\sup_{t_0 \leq t \leq T}\norm{\overline{Y}(t) - Y(t)}} \leq \EE\sbrac{\sup_{t_0 \leq t \leq T}\norm{\overline{Y}(t) - Y(t)}^2}^{1/2} \leq C(h+h^{3/2})e^{C(1+h)T} = \mathcal{O}(h) .
\]
In turn
\[
\EE\sbrac{\sup_{0 \leq k \leq N-1}\norm{Y_k - Y(t_k)}} = \EE\sbrac{\sup_{0 \leq k \leq N-1}\norm{\overline{Y}(t_k) - Y(t_k)}} \leq \EE\sbrac{\sup_{t_0 \leq t \leq T}\norm{\overline{Y}(t) - Y(t)}} = \mathcal{O}(h) ,
\]
which gives \eqref{eq:isgfconsist}. For \eqref{eq:igfconsist}, it is sufficient to see that with the system \eqref{ISIHD-S-product}, the term $\varsigma(s,y(s))$ disappears in the main inequality \eqref{eq:consistmainineq}. Starting from there, this becomes the leading term that scales as $\mathcal{O}(h)$ (instead of $\mathcal{O}(h^2)$ previously) hence entailing \eqref{eq:igfconsist}. 
\end{proof}
}
\section{Auxiliary results}\label{aux}
\subsection{Deterministic results}
\label{detaux}
\begin{lemma}\label{lim0}
Let $t_0>0$ and $a,b:[t_0,+\infty[\rightarrow \R_+$. If $\lim_{t\rightarrow \infty} a(t)$ exists, $b\notin\Lp^1([t_0,+\infty[)$ and $\int_{t_0}^\infty a(s)b(s)ds<+\infty$, then $\lim_{t\rightarrow \infty} a(t)=0.$
\end{lemma}

\begin{lemma}\label{ab}
Let $a,b:[t_0,+\infty[\rightarrow\R_+$ be two functions such that $a\notin \Lp^1([t_0,+\infty[)$, $\lim_{u\rightarrow+\infty}b(u)=0$, and define $A(t)\eqdef\int_{t_0}^{t}a(u)du$ and $B(t)\eqdef e^{-A(t)}\int_{t_0}^t a(u)e^{A(u)}b(u) du$. Then $\lim_{t\rightarrow+\infty}B(t)=0$.
\end{lemma}
\proof{}
Let $\varepsilon>0$ arbitrary, let us take $T_{\varepsilon}$ such that $t_0<T_{\varepsilon}$ and $b(u)\leq\varepsilon$ for $u\geq T_{\varepsilon}$. For $t>T_{\varepsilon}$, let us write \begin{align*}
    B(t)&=e^{-A(t)}\int_{t_0}^{T_{\varepsilon}}a(u)e^{A(u)}b(u)du+e^{-A(t)}\int_{T_{\varepsilon}}^t a(u)e^{A(u)}b(u) du\\
    &\leq e^{-A(t)}\int_{t_0}^{T_{\varepsilon}}a(u)e^{A(u)}b(u)du+\varepsilon.
\end{align*}
Since $a\notin \Lp^1([t_0,+\infty[)$, then $\lim_{t\rightarrow+\infty}e^{-A(t)}=0$, we get $$\limsup_{t\rightarrow+\infty}B(t)\leq \varepsilon.$$
This being true for any $\varepsilon>0$, we infer that $\lim_{t\rightarrow+\infty}B(t)=0$, which gives the claim.
\endproof

\begin{lemma}\label{1gam}
Under hypothesis \eqref{H1}, then $$\int_{t_0}^{\infty} \frac{ds}{\Gamma(s)}=+\infty.$$
\end{lemma}

\proof{}
Let $q(t)\eqdef \int_t^{\infty}\frac{ds}{p(s)}$, since $\int_{t_0}^{\infty}\frac{ds}{p(s)}<+\infty$, then $\lim_{t\rightarrow\infty} q(t)=0$ and $q'(t)=-\frac{1}{p(t)}$. On the other hand \begin{align*}
    \int_{t_0}^{\infty} \frac{ds}{\Gamma(s)}=-\int_{t_0}^{\infty}\frac{q'(t)}{q(t)}=\ln(q(t_0))-\lim_{t\rightarrow\infty}\ln(q(t))=+\infty.
\end{align*}
\endproof

\subsection{On stochastic processes}\label{onstochastic}
Let us recall some elements of stochastic analysis. Throughout the paper, $(\Omega,\mathcal F,\mathbb P)$ is a probability space and $\{{\mathcal F}_t| t\geq 0\}$ is a filtration of the $\sigma-$algebra $\mathcal F$. Given $\mathcal{C}\in 2^\Omega$, we will denote $\sigma(\mathcal{C})$ the $\sigma-$algebra generated by $\mathcal{C}$. We denote $\mathcal F_{\infty}\eqdef \sigma \left(\bigcup_{t\geq 0} \mathcal F_t \right)\in\mathcal F$.

The expectation of a random variable $\xi:\Omega\rightarrow\H$ is denoted by 
\[
\EE(\xi)\eqdef \int_{\Omega}\xi(\omega)d\PP(\omega).
\]
An event $E\in\calF$ happens almost surely if $\PP(E)=1$, and it will be denoted as ``$E$, $\PP$-a.s.'' or simply ``$E$, a.s.''. The indicator function of an event $E\in\calF$ is denoted by 
\[
\ind_E (\omega) \eqdef
\begin{cases}
1 & \text{if } \omega\in E,\\
0 & \text{otherwise}.
\end{cases}
\] 
An $\H$-valued stochastic process starting at $t_0\geq 0$ is a function $X:\Omega\times[t_0,+\infty[\rightarrow\H$. It is said to be continuous if $X(\omega,\cdot)\in C([t_0,+\infty[;\H)$ for almost all $\omega\in\Omega$. We will denote $X(t)\eqdef X(\cdot,t)$. We are going to study SDEs and SDIs, and in order to ensure the uniqueness of a solution, we introduce a relation over stochastic processes. Two stochastic processes $X,Y:\Omega\times [t_0,T]\rightarrow\H$ are said to be equivalent if $X(t)=Y(t)$, $\forall t\in [t_0,T]$, $\PP$-a.s. This leads us to define the equivalence relation $\calR$, which associates the equivalent stochastic processes in the same class. 

\smallskip

Furthermore, we will need some properties about the measurability of these processes. A stochastic process $X:\Omega\times [t_0,+\infty[\rightarrow\H$ is progressively measurable if for every $t\geq t_0$, the map $\Omega\times[t_0,t]\rightarrow\H$ defined by $(\omega,s)\rightarrow X(\omega,s)$ is $\calF_t\otimes\calB([t_0,t])$-measurable, where $\otimes$ is the product $\sigma$-algebra and $\calB$ is the Borel $\sigma$-algebra. On the other hand, $X$ is $\calF_t$-adapted if $X(t)$ is $\calF_t$-measurable for every $t\geq t_0$. It is a direct consequence of the definition that if $X$ is progressively measurable, then $X$ is $\calF_t$-adapted.

\smallskip

Let us define the quotient space:
\[
S_{\H}^0[t_0,T] \eqdef \left\lbrace X: \Omega\times[t_0,T]\rightarrow\H, \; X \text{ is a prog. measurable cont. stochastic process}\right\rbrace\Big/\calR.
\]
Set $S_{\H}^0[t_0]\eqdef \bigcap_{T\geq t_0} S_{\H}^0[t_0,T]$.
For $\nu>0$, we define $S_{\H}^{\nu}[t_0,T]$ as the subset of processes $X(t)$ in $S_{\H}^0[t_0,T]$ such that 
\[
S_{\H}^{\nu}[t_0,T] \eqdef  \left\lbrace X\in S_{\H}^0[t_0,T]:  \;
 \EE\pa{\sup_{t\in[t_0,T]}\norm{X_t}^{\nu}}<+\infty \right\rbrace.
\] 
We define $S_{\H}^{\nu}[t_0] \eqdef \bigcap_{T\geq t_0} S_{\H}^{\nu}[t_0,T]$. 
\smallskip

Let $\K$ be a real separable Hilbert space, and $I\subseteq \N$ be a numerable set such that $\{e_i\}_{i\in I}$ is an orthonormal basis of $\K$, and $\{w_i(t)\}_{i\in I, t\geq 0}$ be a sequence of independent Brownian motions defined on the filtered space $(\Omega,\calF,\calF_t,\Pro)$. The process $$W(t)=\sum_{i\in I} w_i(t)e_i$$
is well-defined (independent from the election of $\{e_i\}_{i\in I}$) and is called a $\K$-valued Brownian motion. Besides, let $G:\Omega\times\R_+\rightarrow\calL_2(\K;\H)$ be a measurable and $\calF_t$-adapted process, then we can define $\int_0^t G(s)dW(s)$ which is the stochastic integral of $G$, and we have that the application $G\rightarrow\int_0^{\cdot}G(s)dW(s)$ is an isometry between the measurable and $\calF_t-$adapted $\calL_2(\K;\H)-$valued processes and the space of $\H$-valued continuous square-integrable martingales \cite[Theorem 2.3]{infinite}.

\bibliographystyle{plain}
\smaller
\bibliography{citas}

\end{document}